\documentclass[11pt,a4paper]{amsart}
\usepackage[utf8]{inputenc}
\usepackage[T1]{fontenc}
\usepackage{amsmath,amsthm,xcolor}
\usepackage{amssymb}
\usepackage{amsfonts}
\usepackage[pdftex]{graphicx}
\usepackage{pgf,tikz,pgfplots}
\usetikzlibrary{arrows}
\usepackage[english]{babel}

\usepackage{physics}
\usepackage{mathdots}
\usepackage{yhmath}
\usepackage{cancel}
\usepackage{color}
\usepackage{siunitx}
\usepackage{array}
\usepackage{multirow}
\usepackage{amssymb}
\usepackage{gensymb}
\usepackage{tabularx}
\usepackage{extarrows}
\usepackage{booktabs}
\usetikzlibrary{fadings}
\usetikzlibrary{patterns}
\usetikzlibrary{shadows.blur}
\usetikzlibrary{shapes}

\theoremstyle{plain}
\newtheorem{theorem}{Theorem}[section]
\newtheorem{lemma}[theorem]{Lemma}
\newtheorem{corollary}[theorem]{Corollary}
\newtheorem{proposition}[theorem]{Proposition}
\theoremstyle{definition}
\newtheorem{definition}[theorem]{Definition}
\newtheorem{example}{Example}
\theoremstyle{remark}
\newtheorem{remark}{Remark}

\newcommand{\N}{\mathbb{N}}
\newcommand{\Z}{\mathbb{Z}}
\newcommand{\Q}{\mathbb{Q}}

\newcommand{\R}{\mathbb{R}}
\newcommand{\T}{\mathbb{T}}

\newcommand {\h}{\mathfrak h}
\newcommand {\D}{\mathcal D}

\begin{document}

\title[Isoperimetric-type ineq. for Mather's $\beta$-function of convex bill.]{Isoperimetric-type  inequalities for Mather's $\beta$-function of convex billiards}

\author{Stefano Baranzini}
\address{Dipartimento di Matematica ``Giuseppe Peano'', Università degli Studi di Torino, Via Carlo Alberto 10, 10123 Torino, Italy}
\email{stefano.baranzini@unito.it}

\author{Misha Bialy}
\address{School of Mathematical Sciences, Raymond and Beverly Sackler Faculty of Exact Sciences, Tel Aviv University, Israel} 
\email{bialy@tauex.tau.ac.il}

\author{Alfonso Sorrentino}
\address{Dipartimento di Matematica, Università degli  Studi di Roma ``Tor Vergata'', Via della ricerca scientifica 1, 00133 Rome, Italy}
\email{sorrentino@mat.uniroma2.it}

\maketitle 

\begin{abstract} 
In this article we discuss pointwise spectral rigidity results for several billiard systems ({\it e.g.}, Birkhoff billiards, symplectic billiards and $4^{\rm th}$ billiards), showing that a single value of Mather’s $\beta$-function can determine whether a strongly convex smooth planar domain is a disk (or an ellipse, in the affine-invariant case of  symplectic billiards). 
Evoking the famous question ``{\it Can you hear the shape of a billiard?}'', one could say that circular billiards can be heard by a single whisper! More specifically, we prove isoperimetric-type inequalities comparing the $\beta$-function associated to the billiard map of domain to that of a disk with the same perimeter or area, and investigate what are the consequences of having an equality.
Surprisingly, this rigidity fails for outer billiards, where explicit counterexamples are constructed for rotation numbers $1/3$ and $1/4$. The results are framed within Aubry–Mather theory and provide a modern dynamical reinterpretation and extension of classical geometric inequalities for extremal polygons.
\end{abstract}

\section{Introduction}

A \textit{Birkhoff billiard} models the motion of a point particle moving at constant speed within a smooth, strictly convex, bounded domain $\Omega \subset \mathbb{R}^2$, whose boundary $\partial \Omega$ is a $C^2$-smooth curve of positive curvature. The particle travels along straight lines and reflects elastically upon hitting the boundary, so that the angle of incidence equals the angle of reflection.\\

	\begin{figure}[h]
	\centering
	\includegraphics[width=0.4\textwidth]{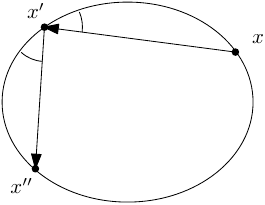}
	\caption{Birkhoff billiard}
	\label{fig:Birkhoff}
\end{figure}

Billiard systems have long fascinated researchers across multiple mathematical disciplines. Their intuitive physical interpretation and conceptual simplicity stand in contrast to their rich and varied dynamical behavior. Moreover, while it is evident that the geometry dictates the motion, a more subtle and compelling question is to which extent the geometry can be reconstructed from the dynamics.  This problem lies at the heart of numerous rigidity phenomena and longstanding conjectures.\\

One of the most iconic questions in this framework, with a long and rich history at the intersection of dynamical systems and geometry, is: \textit{Can one hear the shape of a billiard?} That is, can we recover the shape of the domain from the action of its periodic orbits?
In the context of Birkhoff billiards, this takes a particularly interesting geometric form, as periodic orbits correspond to inscribed polygons, and their action relates to their perimeter.\\

Among all periodic orbits, a special role is played by those maximizing the perimeter in their rotation number class, also known as {\it Aubry–Mather  orbits}; the collection of their actions, marked with the corresponding rotation numbers, is usually called the \textit{marked action} (or \textit{length}) spectrum of the domain. \\

In this article, we focus on the question whether the marked length spectrum determines the domain, in the specific case of circular billiards. We explore this through the lens of Aubry–Mather theory. In this context, an important role is played by  the so-called {\it Mather's $\beta$-function}, which encodes the {\it minimal average action} of orbits with a given rotation number, hence representing a natural extension of the marked length spectrum also to irrational rotation numbers. A more detailed account of this function and its properties will be given in Section \ref{sec:AubryMather}.\\

While it is known that the full $\beta$-function determines whether the domain is a disk (a result that follows from \cite{Bialy}), much less is understood in general, even for other billiards such as elliptic ones, for which the dynamics can be fully described. We summarize known results in Section \ref{sec:globalbeta}. We stress that all these results concern the knowledge of the entire function or infinitely many of its values, and are mostly obtained as corollaries of results on the Birkhoff conjecture (which claims that the only integrable tables are circular and elliptic ones).\\

In this article, we prove the following surprising result:\\

\noindent \textit{A single value of the $\beta$-function of a Birkhoff billiard map suffices to determine whether the domain is a disk}.\\ See  {\bf Theorem \ref{thmBirkbill}} for a complete statement. \\

This {pointwise spectral rigidity} result --which, to our knowledge, is the first of its kind in dynamics --  is established via an isoperimetric-type inequality for Mather’s $\beta$-function, which implies that equality at certain rotation numbers forces the domain to be a circle. Figuratively speaking, one could say that {\it circular billiards can be heard by a single whisper}! More specifically:\\

\noindent \textit{Mather's $\beta$-function of the billiard map in $\Omega$ is always less or equal than Mather's $\beta$-function of the billiard map in a disk of the same perimeter as $\Omega$. If equality holds at some rational in $\left(0, \frac 12\right)$ (or at an irrational not belonging to a certain countable family), then $\Omega$ is a disk.} \\ See  {\bf Theorem \ref{thmBirkbill}} for a complete statement. \\

We remark that  there are indeed non-circular domains (examples of {\it Gutkin billiards}) for which equality holds at some irrational number. Irrational rotation numbers for which this can happen can be explicitely characterized (see Remark \ref{remthmbirkhoff} and \eqref{defR}).\\

We extend this analysis to three natural generalizations of Birkhoff billiards: \textit{outer billiards}, \textit{symplectic billiards}, and \textit{outer length} (or $4^{\text{th}}$) \textit{billiards}. These models are natural variations of the classical one, obtained by combining the following modifications: the dynamics occurs either inside or outside the domain, and the action of their periodic orbits is related to either the perimeter or the area of the corresponding polygon. We summarize their key features in Table \ref{tab:billiard_classification}; we refer to Section \ref{sec:fab4} for a more detailed description of these models.\\

\begin{table}[h]
\centering
\small 
\renewcommand{\arraystretch}{1.8}
\begin{tabular}{|>{\centering\arraybackslash}m{2.8cm}|>{\centering\arraybackslash}m{2.8cm}|>{\centering\arraybackslash}m{2.8cm}|}
\hline
 & \textbf{Dynamics inside $\Omega$} & \textbf{Dynamics outside $\Omega$} \\
\hline
\textbf{Action related to the length} & Birkhoff billiards \break {\Small Section \ref{sec:Birkhoffbill}} & Outer length  (or $4^{\rm th}$) billiards \break {\Small Section \ref{sec:4thbill}}\\
\hline
\textbf{Action related to the area} & Symplectic billiards \break {\Small Section \ref{sec:symplbill}} & Outer (or dual) billiards \break  {\Small Section \ref{sec:outerbill}}\\
\hline
\end{tabular}
\bigskip
\caption{Classification of the four billiard models}
\label{tab:billiard_classification}
\end{table}

In particular, we prove that:\\

\noindent \textit{A single value of the $\beta$-function of a symplectic or $4^{\rm th}$ billiard map suffices to determine whether the domain is a disk}. \\
See  {\bf Theorems \ref{thmsymplbill}} and {\bf \ref{thm4thbill}} for  complete statements. \\

Similarly as above, also in these cases we establish isoperimetric-type inequalities for Mather’s $\beta$-function. More specifically:\\

\noindent \textit{Mather's $\beta$-function of the billiard map in $\Omega$ is always less or equal than Mather's $\beta$-function of the billiard map in a disk of the same area} (respectively, {\it perimeter}) {\it as $\Omega$. If equality holds at some number in $\left(0, \frac 12\right)$, then $\Omega$ is a disk}. \\
See  {\bf Theorems \ref{thmsymplbill}} and {\bf \ref{thm4thbill}} for  complete statements. \\

We remark that in these cases:
\begin{itemize}
\item the equality case for any number in $\left(0, \frac 12\right)$, not necessarily rational, allows us to deduce that $\Omega$ is a disk;
\item symplectic billiards (and outer billiards) are affine-invariant, hence disks and ellipses are dynamically equivalent; hence, the above rigidity result should be interpreted up to affine equivalence.\\
\end{itemize}

 Key steps in our proofs of these results are essentially the careful choice of a (non-standard) generating function and/or a clever parametrization of the billiard table's boundary.\\

So far, these four billiard models appear to share similar rigidity properties, particularly with regard to integrability (e.g., the Birkhoff conjecture) and their action spectra. In essence, many of the rigidity results known for Birkhoff billiards have found natural extensions to these generalized settings. However, no genuinely new or previously unknown phenomena -- distinct from those observed in classical Birkhoff billiards -- have yet emerged.\\

Surprisingly, the pointwise spectral rigidity described above does not extend to outer billiards, for which we construct explicit counterexamples --- for rotation numbers $1/3$ and $1/4$ --- to the isoperimetric-type inequality underlying the rigidity result. See {\bf Theorem \ref{thm:beta1/3}} (rotation number $1/3$) and {\bf Theorems \ref{thm:14centralsym} \& \ref{thm:nonsymmetric1/4}} (rotation number 1/4).

This will be achieved by exploiting some relations between the dynamics of outer  and symplectic billiards, in presence of invariant curves consisting of periodic points of rotation number $\frac 13$ and $\frac 14$. \\
 This distinction may be viewed as one of the first observed spectral differences between outer billiards and the other three models.\\

\subsection{Comparison with previous literature} \label{seccomparison}
These results can be viewed as dynamical revisitation of classical -- and in some cases not very well-known outside certain communities -- inequalities from convex geometry concerning extremal polygons, which were recently rediscovered also by Aliev (for Birkhoff billiards, see \cite{Aliev}).\\

 More specifically, for the perimeter of inscribed and circumscribed $n$-gons  the following results were proven by Schneider in 1971  \cite{Schneider1,Schneider2}. The equality case in the first result was proven partially by Schneider and completed by Florian and Prachar in 1986, \cite{FlorianPrachar} (see also the book Fejes T\'oth \cite{Fejesetco} p.196). \\

\begin{itemize}
\item[{\bf (i)}] {\it Let $\mathcal L^{\rm ins}_{\Omega}\big(n\big)$ be the maximal perimeter of  inscribed $n$-gons in $\Omega$, $n\geq 3$, and  denote by $|\partial\Omega|$ the length of $\partial\Omega$. Then:
\begin{equation}\label{insperimeter}
\mathcal L^{\rm ins}_{\Omega}\left(n \right)\geq \frac{n\, |\partial\Omega|}{\pi}\,\sin\bigg(\frac{\pi}{n}\bigg).
\end{equation}
If for some $n\geq 3$ equality is achieved, then $\Omega$ is a disk.}\\
\item[{\bf (ii)}] {\it Let $\mathcal L_{\Omega}^{\rm circ}\left(n \right)$ be the minimal perimeter of circumscribed $n$-gons of $\Omega$, $n\geq 3$, and denote by $|\partial\Omega|$ the length of $\partial\Omega$. Then:
\begin{equation}\label{circperimeter}
\mathcal L_{\Omega}^{\rm circ}\left(n \right)\leq \frac{n\, |\partial\Omega|}{\pi}\,
\tan\bigg(\frac{\pi}{n}\bigg).
\end{equation}
If for some $n \geq 3$ equality is achieved, then $\Omega$ is a disk.}\\
\end{itemize}

As for the areas of inscribed polygons, the following  was proven by Sas in 1939 (see \cite{Sas}):

\begin{itemize}
\item[{\bf (iii)}]
{\it Let $\mathcal A^{\rm ins}_\Omega(n)$ be the maximal area of $n$-gons  inscribed in $\Omega$, $n\geq 3$, and denote by $|\Omega|$  the area of $\Omega$. Then:
\begin{equation}\label{insarea}
		\mathcal A^{\rm ins}_\Omega\left(n\right)
		\geq \frac{n\, |\Omega|}{2\pi}\, \sin\bigg(\frac{2\pi}{n}\bigg).
\end{equation}
If for some $n\geq 3$ equality is achieved, then $\Omega$ is an ellipse.}\\
\end{itemize}

In this article, we revisit these inequalities for the corresponding $\beta$-functions  through a more modern dynamical and variational lens. This allows us  to extend the analysis of the equality case also  to  non-rational rotation numbers,  and to highlight a relation with so-called {\it Gutkin's billiards} (see Remark \ref{remthmbirkhoff} (ii)). 
In some sense, our approach extends one of the paper \cite{AlbersTabachnikov_Dowker} where Dowker's theorems were understood via Mather's $\beta$-function.\\
Moreover, we investigate the failure of similar inequalities in the case of outer billiards (which correspond to circumscribed $n$-gons of minimal area). \\

\subsection{Organization of the article}
The article is organized as follows:
\begin{itemize}
\item In Section \ref{sec:AubryMather} we recall the definition of Mather's $\beta$-function and recall some of its properties. In particular, we explain the main idea behind our results discussing a simple example, namely a pointwise rigidity result for $\beta$-funtions associated to integrable twist maps perturbed by a potential (see Theorem \ref{thm1}).
\item In Section \ref{sec:fab4} we introduce the four billiard models and specify the  generating functions that we are going to use,  providing a geometric interpretation of the corresponding action of  periodic orbits.
\item Section \ref{secrigiditybeta} is the core of the article. After having recalled what is known about global rigidity of the $\beta$-functions of these billiard models (Section \ref{sec:globalbeta}), in Section \ref{sec:mainresults}  we state our main pointwise rigidity results and provide their proofs in the subsequent sections.
\item In Section \ref{sec:counterexouterbill} we show that these local rigidity results do not extend to outer billiards and construct examples, relating the dynamics of symplectic and outer billiards in presence of invariant curves consisting of periodic points of rotation number $1/3$ (Section \ref{sec:ex13}) and $1/4$ (Section \ref{sec:14}). We conclude the section with some open questions that naturally arise and, in our opinion, deserve further investigation (Section \ref{sec:openquestions}).
\end{itemize}

\bigskip

\subsection*{Acknowledgements}
MB and AS are grateful to the organizers of the workshop ``{\it Billiards and quantitative symplectic geometry}'', held at Universit\"at Heidelberg (14 - 18 July 2025), for the very stimulating environment.\\ 
MB acknowledges the  supported of ISF grant 974/24. SB and AS acknowledge the support of the Italian Ministry of University and Research’s PRIN 2022 grant ``{\it Stability in Hamiltonian dynamics and beyond}'' and of the 2025 INDAM-GNAMPA project ``{\it Non integrabilità e complessità in Meccanica
Celeste}''. AS also acknowledges the support of the Department of Excellence grant MatMod@TOV (2023-27), awarded to the Department of Mathematic at University of Rome Tor Vergata. \\
SB and AS are members of the INdAM research group GNAMPA and the UMI group DinAmicI.

\section{Aubry-Mather theory for twist maps of the annulus}\label{sec:AubryMather}
In this section we provide a coincise introduction to Aubry-Mather theory and introduce our main object of investigation, namely {\it Mather's $\beta$-function}.\\

At the beginning of 1980s Serge Aubry \cite{Aubry, AubryLeDaeron} and John Mather \cite{Mather1982} developed, independently, what nowadays is commonly called {\it Aubry--Mather theory}.
This novel approach to the study of the dynamics of twist diffeomorphisms of the annulus, pointed out the existence of many {\it action-minimizing orbits}  for any given rotation number. 
For a more detailed introduction,  see  for example \cite{MatherForni, Siburg, SorLecNotes}).\\

More precisely, let $a,b \in \R $, with $a<b$, and let 
$$f: \R/\Z \times (a,b) \longrightarrow \R/\Z \times (a,b)$$ 
be  a {\it positive  twist map}, {\it i.e.}, a $C^1$ diffeomorphism such that its lift  to the universal cover $\tilde{f}$ satisfies the following properties (we denote $(x_1,y_1)= \tilde{f}(x_0,y_0)$):
\begin{itemize}
	\item[(i)] $\tilde{f}(x_0+1, y_0) = \tilde{f}(x_0, y_0) + (1,0)$;
	\item[(ii)] $\tilde{f}$ extends continuously to $\R\times \{a\}$ and $\R\times \{b\}$ by a rotation:
	$$
	\tilde{f}(x,a) = (x+ \omega_-, a) \qquad {\rm and }\qquad \tilde{f}(x,b) = (x+ \omega_+, b).
	$$
	\item[(iii)] $\frac{\partial x_1}{\partial y_0} \geq c>0$  (positive  twist condition),
	\item[(iv)] $\tilde{f}$ admits a (periodic) {\it generating function} $S$ ({\it i.e.},  it is an exact symplectic map):
	$$
	y_1\,dx_1 - y_0\,dx_0 = dS(x_0,x_1).\\
	$$
\end{itemize}

\medskip

We call the  interval $(\omega_-,\omega_+)\subset \R$  the twist interval of $f$ (notice that $\omega_- < \omega_+$ because of the positive twist condition).

\begin{remark}
(i) The definition of negative  twist map is similar, but condition  (iii) is replaced by $\frac{\partial x_1}{\partial y_0} \leq c < 0$.\\
(ii) One could also consider the infinite (or semi-infinite) cylinder, {\it i.e.}, $a$ or $b$ might be infinite, with an easy adaptation of the definition. In this case the twist interval will be an infinite interval. 
\end{remark}

In particular, it follows from (iv) that:

\begin{equation} \label{genfuncttwistmap}
	\left\{
	\begin{array}{l}
		y_1 = \partial_2 S (x_0,x_1)\\
		y_0 = - \partial_1 S(x_0,x_1)\,,\\
	\end{array}
	\right.
\end{equation}
where $\partial_i$, $i=1,2$, denotes the partial derivatives with respect to the $i$-th component.\\

\medskip

As it follows from (\ref{genfuncttwistmap}), orbits $(x_i)_{i\in\Z}$ of the  twist map $f$ correspond to  critical configurations  of the {\it action functional}
$$
\{x_i\}_{i\in\Z} \longmapsto \sum_{i\in \Z} S(x_i, x_{i+1})
$$
and vice-versa.\\

Aubry-Mather theory is concerned with the study of orbits that minimize this action-functional amongst all configurations with a prescribed rotation number; recall that the rotation number of an orbit $\{x_i\}_{i\in\Z}$ is given by $ \omega = \lim_{i\rightarrow \pm \infty} \frac{x_i}{i}$, if this limit exists. In this context, {\it minimizing} is meant in the statistical mechanical sense, {\it i.e.}, every finite segment of the configuration minimizes the action functional with fixed end-point condition.\\

\noindent {\bf Theorem (Aubry \cite{Aubry, AubryLeDaeron}, Mather \cite{Mather1982, MatherForni})}  {\it A positive twist map possesses action-minimizing orbits for every rotation number in its twist interval $(\omega_-,\omega_+)$. For every rational  rotation number in the twist interval, there is  at least one action-minimizing periodic orbit.  Moreover, every action-minimizing orbit lies on a Lipschitz graph over the $x$-axis.}\\

We can now introduce the {\it minimal average action} (or {\it Mather's $\beta$-function}).

\begin{definition}\label{defbeta}
	Let $x^{\omega} = \{x_i\}_{i\in\Z}$ be any minimal orbit with rotation number $\omega$. Then, the value of the {\em minimal average action} at $\omega$ is given by (this value is well-defined, since it does not depend on the chosen orbit $x^\omega$):
	\begin{equation}\label{avaction}
		\beta(\omega) := \lim_{N\rightarrow +\infty} \frac{1}{2N} \sum_{i=-N}^{N-1} S(x_i,x_{i+1}).
	\end{equation}
\end{definition}

This function $\beta: (\omega_-, \omega_+) \longrightarrow \R$ enjoys many properties and encodes interesting information on the dynamics. In particular:
\begin{itemize}
	\item[i)] $\beta$ is strictly convex and, hence, continuous (see \cite{MatherForni});
	\item[ii)] $\beta$ is differentiable at all irrationals (see \cite{Mather90});
	\item[iii)] $\beta$ is differentiable at a rational $p/q$ if and only if there exists an invariant circle consisting of periodic minimal orbits of rotation number $p/q$ (see \cite{Mather90}).\\
\end{itemize}

For irrational $\omega$, a more useful characterization of $\beta$ is the following:
$\beta(\omega)$ is the minimum of the so-called {\it Percival's Lagrangian}
\[
P_\omega(\psi) := \int_0^1 S(\psi(t), \psi(t+\omega)) \, dt
\]
 over the set of measurable functions $\psi:\R \to \R$ that satisfy the periodicity condition $\psi(t + 1) = \psi(t) + 1$
and the condition that $\psi(t) - t$ is bounded.
Moreover, the minimizer of $P_\omega$ is unique up to translation, {\it i.e.}, if $\varphi$ and $\tilde{\varphi}$ minimize $P_\omega$, then $\tilde{\varphi}(t) = \varphi(t + a)$ almost everywhere,
for some $a \in \mathbb{R}$ (see \cite[Theorem 14.3]{MatherForni}).\\
In particular, if $\varphi$ minimizes $P_\omega$, then we may define  action-minimizing configurations of rotation number $\omega$ as follows (see \cite[Section 12]{MatherForni}). For every $t \in \mathbb{R}$, we set
\begin{equation}\label{volotea}
x^{t\pm 0}_n := \varphi(t + n\,\omega  \pm 0) \qquad n\in \Z.
\end{equation}

\medskip

\begin{remark}
Actually, also a sort of converse is true. Starting from an action-minimizing configuration $x=\{x_n\}_{n\in\Z}$ with irrational rotation number $\omega$, it is possible to obtain a minimizer of $P_\omega$ in the following way.
Let us consider the so-called  \emph{Aubry's hull function} associated to $x=\{x_n\}_{n\in\Z}$:
\[
\varphi_x(t) = \sup\{ x_{-q} + p: \;  p - q \omega \leq t \}.
\]
One can prove that $\varphi_x: \R\longrightarrow \R$ is strictly increasing and $\varphi_x(t + 1) = \varphi_x(t) + 1$ (see for example \cite[Section 12]{MatherForni}).
Moreover,  $\varphi_x$ minimizes $P_\omega$.\\
\end{remark}

We can now prove the following lemma that will play a crucial role for our results.

	\begin{lemma}\label{linearconfigurations}
Let  $x_n:=x_0+n\,\omega$, $n\in\Z$,  be an $\omega$-equispaced configuration sequence with the rotation number $\omega$ and starting point at $x_0 \in [0,1)$.
Then, its average action $\mathcal A_\omega(x_0):=\lim_{N\rightarrow +\infty} \frac{1}{2N} \sum_{i=-N}^{N-1} S(x_i,x_{i+1})$ satisfies the inequality:
\begin{equation}\label{eq:Abetaomega}
\mathcal A_\omega(x_0) \geq\beta(\omega).
\end{equation}
Moreover, if  equality holds then the linear configuration $\{x_n\}_{n\in\Z}$ is an action-minimizing configuration. 
In particular, if $\omega$ is irrational, 
the twist map admits an invariant circle of rotation number $\omega$ given by $\{y=\omega\}$ (hence, all $\omega$-equispaced configurations are action-minimizing). 
	\end{lemma}

	\begin{proof}
	The claim is easy for rational rotation numbers.
	Let us  consider the case of $\omega$  irrational.
	Let $\{y_n\}_n$ be an action-minimizing configuration of rotation number $\omega$; observe that $|y_n - y_0 - n\, \omega| < 1$ for every $n \in \Z$ (see for example \cite[Lemma 9.1 and Corollary 10.3]{Gole}). Then, it follows from the definition of action-minimizing configuration that for every $N\geq 1$:
$$
\sum_{i=-N}^{N-1}S(y_i,y_{i+1}) \leq \sum_{i=-N}^{N-1} S(x_i,x_{i+1}) + C,
$$
where $C>0$ is a constant independent on $n$ (it keeps  into account the correction due to the fact that a-priori $x_{\pm N} \neq y_{\pm N}$). Taking the average and the limit as $N\longrightarrow +\infty$, yield inequality \eqref{eq:Abetaomega}.

Let us now focus on the equality case, {\it i.e.}, $\mathcal A_\omega(x_0)=\beta(\omega)$. 
By Birkhoff ergodic theorem, we have:
\begin{equation}\label{indepx0}
\mathcal A_\omega(x_0)=\int_{0}^{1} S(t,t+\omega)dt. 
\end{equation}
Recalling now that $\beta(\omega)$ corresponds to the minimal value of Percival Lagrangian $P_\omega$ (see above), we deduce that if
 $\mathcal A_\omega (x_0) = \beta(\omega)$,  then $\varphi(t)=t$ is  a minimizer of $P_\omega$. Then, for every 
$t \in \mathbb{R}$ we have (see \eqref{volotea}, using the fact  that $\varphi$ is continuous):
\[
x^{t}_n := \varphi (t + n\,\omega) = t+n\,\omega \qquad n\in \Z
\]
is  action-minimizer with rotation number $\omega$. In particular, it follows that for every $x_0$ the corresponding $\omega$-equispaced configuration is action-minimizing and
this completes the proof \text.

	\end{proof}

\medskip

\subsection{Toy example: integrable twist maps}	\label{toyexample}	
Let us now discuss a simple example and prove a pointwise rigidity result of Mather's $\beta$-function associated to integrable twist maps. The main goal of this example is to explain, in a very simple setting, the main ideas that we will exploit in the billiard case.\\

Let us consider a  \emph{completely integrable} symplectic twist map  of the cylinder $\T\times \R$, where $\T := \R/\Z$:
    \begin{equation}\label{ECI}
f_0 (q, p):=(q+ \h'(p), p),
\end{equation} 
where $\h:\R\rightarrow \R$ is $C^2$ and such that $\nabla\h:\R\to\R$ is  a $C^1$ diffeomorphism. It is easy to check that a generating function of \eqref{ECI} is given by $S_0(q,Q) := \ell(Q-q)$,  where $\ell'=(\h')^{-1}$.

For every $c \in \R$, $f_0$ admits an invariant curve  $\T\times \{c\}$ on which the dynamics is a rotation by $\rho:=\h'(c)$. In particular:
\begin{itemize}
\item if $\h'(c)\in \Q$, every point on $\T\times \{c\}$ is periodic;
\item if $\h'(c)\not \in \Q$, every orbit on $\T\times \{c\}$ is dense.
\end{itemize}

Each orbit is action-minimizing  and their average action is $\ell(\rho)$, where $\rho$ is the corresponding rotation number. Hence, Mather's $\beta$ function is given by
\begin{eqnarray*}
\beta_0: \R &\longrightarrow& \R\\
\rho &\longmapsto& \ell(\rho).
\end{eqnarray*}

Let us now consider a {\it perturbation of $f_0$ by a potential $V$}, namely the  twist map
\begin{equation}
f_V (q, p) :=   f_0 (q, p+ V'(q))
\end{equation}
with generating function
$S_V(q,Q):=S_0(q, Q)+ V(q) = \ell (Q-q) + V(q)$.
Observe that it is not restrictive to assume that $\int_0^1 V(q)\, dq=0$, since adding constant to $S_V$ does not change the dynamics.
We denote by $\beta_V : \R \longrightarrow \R$ the corresponding Mather's $\beta$-function.

\begin{theorem}\label{thm1}
Assume that $\int_0^1 V(q)\, dq=0$. The following inequality holds true:
	$$
	\beta_V(\rho)\leq \beta_0(\rho) \qquad \forall\; \rho\in \R.
	$$
	Moreover, if for some rotation number $\rho \in \R$ equality is achieved, then $V\equiv 0$.
\end{theorem}
\begin{proof}
Let $\rho \in \R$. For every $q\in \R$ let us consider the sequence starting at $q$ and of rotation number $\rho$ given by
$$
q_k := q+ k \rho \qquad k\in \N.
$$
Its average action is given by
\begin{eqnarray*}
\mathcal A^{\rho}_V(q) &:=& \lim_{N\rightarrow +\infty} \frac{1}{N}\sum_{k=0}^{N-1} S_V (q_k, q_{k+1})\\
 &=& \lim_{N\rightarrow +\infty} \frac{1}{N} \sum_{k=0}^{n-1} \left( \ell (q_{k+1}-q_k) + V(q_k) \right)\\
&=&  \ell(\rho) + \lim_{N\rightarrow +\infty} \frac{1}{N} \sum_{k=0}^{N-1} V(q + k \rho)\\
&=&  \ell(\rho) +  \int_0^1 V(z) d\mu_{\rho,q},
\end{eqnarray*}
where  
$$
\mu_{\rho,q} = \left\{
\begin{array}{ll}
{{\rm Leb}_{[0,1]}} & \mbox {if $\rho \not \in \Q$} \\
\\
\frac{1}{n} \sum_{k=0}^{n-1} \delta_{q+ k\frac mn} & \mbox {if $\rho = \frac mn \in \Q \setminus \{0\}$}\\
\\
\delta_q & \mbox{if $\rho=0$}
\end{array} \right.
$$
($\delta_z$ denotes Dirac's delta measure on the point $z$).\\

We take its average:
$$
 \int_0^1 \mathcal A^{\rho}_V(q)\, dq \;= \; \ell(\rho) +  \int_0^1 dx \int_0^{1} V(z) \, d\mu_{\rho,q} \;=\; \ell(\rho),
$$
where we have used that $\int_0^1 V(q) dq = 0$.
This implies that there exist $\bar q\in \R$ such that $A^{\rho}_V(\bar q) \leq \,\ell(\rho)$ and therefore
$$
\beta_V(\rho)  \leq \ell(\rho ) = \beta_0 (\rho).\\
$$

\medskip

Assume now that there exists $\rho_0 \in \R$ such that equality holds, {\it i.e.}, $\beta_V(\rho_0)   = \beta_0 (\rho_0)$.  Then, if $\rho$ is irrational, it follows from Lemma \ref{linearconfigurations} that every $\rho$-equispaced configuration  $\{q_k\}_k$, with $q_k = q+ k\rho$, is action minimizing. Therefore:
$$
\partial_1 S_V (q_k, q_{k+1}) + \partial_2 S_V (q_{k-1}, q_{k}) = 0 \qquad \forall k.
$$
The above equality reads:
\begin{eqnarray*}
- \ell'(q_{k+1}-q_k) + V'(q_k) + \ell'(q_{k}-q_{k-1}) = 0  & \Longleftrightarrow& 
- \ell'(\rho) + V'(q+k\rho) + \ell'(\rho) = 0\\
&\Longleftrightarrow& V'(q+k\rho) = 0 \qquad \forall\,k.
\end{eqnarray*}
Since this is true for all $q\in \R$ and for all $k$, then this implies that $V$ is constant. Since $V$ has zero average, then $V\equiv 0$ in this case.

If $\rho$ is rational, since we have established that
\[
 \beta_V (\rho) \le \int_0^1 \mathcal{A}_V^\rho (q)dq \le \beta_0(\rho)
\]
it follows that $\mathcal{A}_V^\rho (q) = \beta_V(\rho)$ for all $q \in \mathbb{R}$ and therefore all $\rho$-equispaced trajectories are minimal. By the same argument as in the case of irrational $\rho$, we infer that $V\equiv 0$ in this case as well.
\end{proof}

\begin{remark}
The above result continues to hold for higher-dimensional completely twist maps, namely maps of the form:
$$f_0 (q, p):=(q+\nabla \h(p), p),$$
where $\h:\R^d\rightarrow \R$ is $C^2$ and such that $\nabla\h:\R^d\to\R^d$ is  a $C^1$ diffeomorphism. In this case, their generating function  is given by $S_0(q,Q) := \ell(Q-q)$,  where $\nabla \ell=(\nabla h)^{-1}$.\\
\end{remark}

\section{The fab four: four models of  billiards}\label{sec:fab4}

Let $\Omega\subset\R^2$ be a strictly convex domain  with a $C^2$ oriented boundary $\partial \Omega$. Let also $O$ denote the origin of $\mathbb{R}^2$ which we assume to lie inside $\Omega$.\\

We will recall here four models of  planar billiards. See \cite{AlbersTabachnikov_Dowker} for further details.
These models differ by where the dynamics takes place, which can be either inside or outside the domain, and by the action of their periodic orbits, which is related to either the perimeter or the area of the corresponding polygon. See Table \ref{tab:billiard_classification}.

\subsection{Birkhoff billiards: Inner length billiards}\label{sec:Birkhoffbill}
{\it Birkhoff billiards} are the most classical billiards introduced and studied by Birkhoff. Birfhoff billiard map, denoted $B_\Omega$, acts on the space of oriented lines intersecting the domain $\Omega$ according to the reflection rule of geometric optics, angle of reflection equals angle of incidence (see Figure \ref{fig:Birkhoff}).  It is a twist map with the generating function \[
L(x,x')=-|\gamma(x')-\gamma(x)|.
\] 
Notice that the minus sign is introduced here in order to have a positive twist map, as described in Section \ref{sec:AubryMather}. 

\subsubsection{Non-standard generating function}

For our purposes we will need the non-standard generating function introduced in \cite{BialyMironov, BialyMironovAdvances}. More specifically: consider the coordinates $(p, \phi)$ on the space of all oriented lines, where $\phi$ is the angle between the right unit normal to the line and the horizontal axis, and $p$ is the signed distance to that line. With this setup, the space of oriented lines is identified with $T^*\mathbb{S}^1$ (where $\mathbb{S}^1 := \mathbb{R}/2\pi\mathbb{Z}$), equipped with the symplectic form $d\alpha$, where $\alpha := p\,d\phi$. In these coordinates, the boundaries of the phase cylinder where $B_\Omega$ is defined are given by:
\[
\{(p, \phi) : p = h(\phi)\} \quad {\rm and} \quad \{(p, \phi) : p = -h(\phi + \pi)\},
\]
where $h$ denotes the {\it support function} of $\gamma$ with respect to the origin. In particular, one can show that:
\begin{equation}
	\gamma (\phi) = h(\phi)\, (\cos \phi, \sin \phi) + h'(\phi)\, (- \sin \phi, \cos \phi).\\
	\label{eq:envelope}
	\end{equation}

The billiard map $B_\Omega$ is the symplectic map with generating function 
\begin{equation}
	\mathcal{S}^{\tiny\mathrm{int}}_{\Omega}(\phi_0, \phi_1) := -2\, h\left(\frac{\phi_1 + \phi_0}{2}\right)\, \sin\left(\frac{\phi_1 - \phi_0}{2}\right).
	\label{generatingfunction-birkhoff}
\end{equation}
 Notice that also here  the $-$ sign is introduced in formula (\ref{generatingfunction-birkhoff}) in order to get that  $B_\Omega$ is a positive twist map.
We denote by $\beta_\Omega$  the Mather's $\beta$-function associated to the billiard map in $\Omega$.\\

Let $q\geq 2$ and consider a  $q$-periodic trajectory  for $B_\Omega$ winding $p$ times about $\Omega$ with $p,q$ coprime. The number  $\frac pq \in \Q\cap \left(0,\frac 12\right]$ corresponds to the {rotation number} of the periodic orbit. 
It can be checked that the action of such a periodic orbit is
$$
\mathcal A_{\frac pq}=\sum_{k=0}^{{q-1}}	\mathcal{S}^{\tiny\mathrm{int}}_{\Omega}(\phi_k,\phi_{k+1})
$$
is the same for the standard and the non-standard generating function (see for instance \cite[Theorem 2]{bialy_tab_dan_reznik}) and equals the negative of the perimeter of the corresponding $q$-gon of (ordered) vertices $\gamma(\phi_k)$, $k=0, \ldots, q-1$.
So, we have
$$
\beta_\Omega\bigg(\frac pq\bigg)=-\frac1q \mathcal L^{\rm ins}_{\Omega}\bigg(\frac pq\bigg),
$$where $\mathcal L^{\rm ins}_{\Omega}$ is the maximal perimeter of the inscribed $q$-gons with winding number $p$. 
\begin{example}
For the unit disk $\mathcal D$ we have 
$\beta_\mathcal D(\rho)= -2\sin (\pi\rho)$. Indeed, any billiard trajectory inside a disk make a constant angle, say $\pi \rho$, with the tangent to the boundary. This implies that any two consecutive point of an orbit determine a segment of length $ 2 \sin(\pi \rho )$.
\label{example1}
\end{example}

\bigskip

\subsection{Outer billiards: Outer area billiards} \label{sec:outerbill}
The \textit{outer}  {\it billiard} associated to $\Omega$ is defined as follows (see Figure \ref{figure:dual_billiard}). For any point $p\in\R^2\setminus\Omega$, there are two tangent lines to $\partial\Omega$ passing through $p$. Consider the unique one that is tangent to $\partial\Omega$ at a point $q$ and such that the vector $\vec{pq}$ has the same orientation as the boundary $\partial\Omega$ at $q$. Define the image of $p$ by the outer billiard map as the point $B^{\rm out}_\Omega(p)$ on the latter tangent line $T_q\partial\Omega$ such that $q$ is the midpoint between $p$ and $B^{\rm out}_\Omega(p)$ (see Figure \ref{figure:dual_billiard}).\\

\begin{figure}[!h]
	\definecolor{uququq}{rgb}{0.25,0.25,0.25}
	\definecolor{qqqqff}{rgb}{0,0,1}
	\definecolor{xdxdff}{rgb}{0.49,0.49,1}
	\definecolor{cqcqcq}{rgb}{0.75,0.75,0.75}
	\begin{tikzpicture}[line cap=round,line join=round,>=triangle 45,x=3.0cm,y=3.0cm]
		\clip(-1.4,-0.6) rectangle (1,1.4);
		\draw [rotate around={26.57:(0,0.25)}] (0,0.25) ellipse (2.76cm and 2.19cm);
		\draw (-1.16,0.56)-- (0.67,1.31);
		\begin{scriptsize}
			\draw[color=black] (-0.47,0.05) node {$\Omega$};
			\draw[color=black] (-0.28,1.02) node {$q$};
			\fill [color=black] (-1.16,0.56) circle (1.5pt);
			\draw[color=black] (-1.26,0.65) node {$B^{\rm out}_\Omega(p)$};
			\fill [color=black] (0.67,1.31) circle (1.5pt);
			\draw[color=black] (0.71,1.36) node {$p$};
		\end{scriptsize}
	\end{tikzpicture}
	\caption{Outer billiard}
	\label{figure:dual_billiard}
\end{figure}
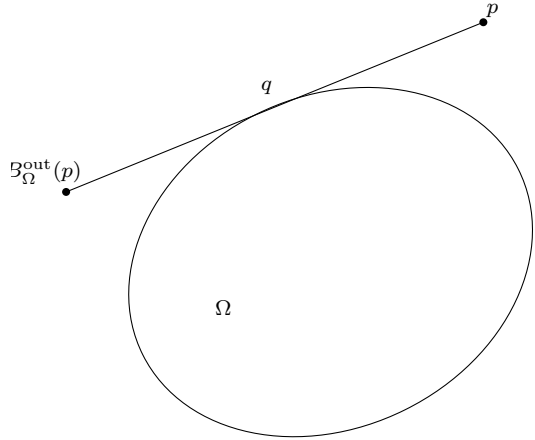

Outer billiards were introduced by Neuman in 1959 \cite{neumann} (in a paper with a rather unusual title), but an earlier construction is due to Day in \cite{day}. Moser popularized this system in the 1970's as a toy model for celestial mechanics:  in some sense, the orbit of a point around the  billiard table resembles the orbit of a celestial body around a planet or a star (see \cite{Moser, Moserbook}).\\

The map $B^{\rm out}_\Omega$ is an area preserving map of the exterior of $\Omega$ with respect to the standard area of the plane, which extends as the identity to $\partial \Omega$. The phase space is foliated by the positive tangent rays to $\partial \Omega$, and $B^{\rm out}_\Omega$ is a positive twist map  with respect to this foliation (see  \cite{boyland, Taba} for more details).
Moreover, $B^{\tiny \rm out}_\Omega$ commutes with affine transformations of the plane; namely, if $A$ is an affine transformation of the plane, then:
$$ B^{\tiny \rm out}_{A(\Omega)} \circ  A =  A \circ B^{\tiny \rm out}_\Omega.$$

\medskip

\subsubsection{Generating functions for outer billiards.}
\label{subsec:generating_function_outer}
There is a ``standard'' generating function $S$, for the outer billiard map corresponding to the so called envelope coordinates $(\lambda,\vartheta)$ \cite{boyland, gk}. Let us consider a parametrization  of $\partial \Omega$
\begin{eqnarray*}
	\gamma: \R/2\pi \Z &\longrightarrow& \R^2\\
	\vartheta &\longmapsto& \gamma(\vartheta)
\end{eqnarray*}
where $\vartheta$  denotes the angular coordinate on $\partial \Omega$, that is, the direction of the right normal of $\partial \Omega$.
Then, every point $M$ in the exterior of $\Omega$ can be uniquely written as 
$\gamma(\vartheta) +\lambda \dot\gamma(\vartheta)/|\dot\gamma (\vartheta)|$ for some $\vartheta\in \R/2\pi \Z$ and $\lambda>0$. 
In these coordinates, the outer billiard map about $\Omega$ is  a  map from the half-cylinder to itself, given by
\begin{eqnarray}\label{Boutmap}
	B^{\tiny \rm out}_\Omega:  \R/2\pi  \Z \times (0,+\infty) &\longrightarrow& \R/2\pi \Z \times (0,+\infty) \nonumber\\
	(\vartheta, \lambda) &\longmapsto& (\vartheta_1,\lambda_1)
\end{eqnarray}
where $(\vartheta_1,\lambda_1)$ is uniquely determined by the condition
$$
\gamma(\vartheta) +\lambda \dot{\gamma}(\vartheta)/|\dot\gamma(\vartheta)| = \gamma(\vartheta_1) -\lambda_1 \dot{\gamma}(\vartheta_1)/|\dot{\gamma}(\vartheta_1)|.$$

In these coordinates  $S$ can be written  by:
\begin{eqnarray*}
	S(\vartheta_0,\vartheta_1): \{(\vartheta_0,\vartheta_1)\in \R^2:\; 0<\vartheta_1-\vartheta_0<\pi\} &\longrightarrow& \R \\
	(\vartheta_0,\vartheta_1) &\longmapsto&S(\vartheta_0,\vartheta_1),
\end{eqnarray*}
where $ S(\vartheta_0,\vartheta_1)$ denotes the area of the (oriented) curvilinear triangle bounded by the tangent lines at $\gamma(\vartheta_0)$ and $\gamma(\vartheta_1)$ and $\partial \Omega$ (see \cite{boyland, DogruTabachnikov, Taba} for more details). In particular, if $(\vartheta_1, \lambda_1) =  B^{\tiny \rm out}_\Omega (\vartheta_0,\lambda_0)$ then:
$$
\partial_1 S (\vartheta_0,\vartheta_1) =  - \frac{\lambda_0^2}{2}, \qquad \partial_2 S (\vartheta_0,\vartheta_1)  = \frac {\lambda_1^2}{2}
$$
where $\partial_i$ denotes the partial derivative with respect to the $i$-th component. Thus the standard symplectic structure of $\R^2$ invariant under $B^{\tiny \rm out}_\Omega$  in these coordinates takes the form
$$
\omega=d(\lambda^2/2)\wedge d\vartheta.
$$

Moreover, the  twist condition holds (it is a positive twist map):
$$
\partial^2_{12} S(\vartheta_0,\vartheta_1)   < 0. \\
$$

\begin{figure}[h]
	\centering
	\includegraphics[width=0.4\textwidth]{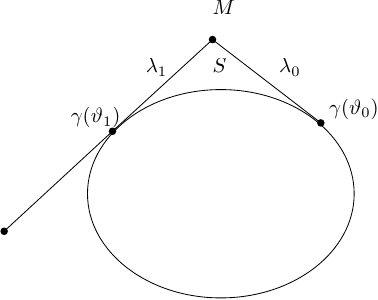}
	\caption{Envelope coordinates}
	\label{fig:outer-rule}
\end{figure}

For our purposes, another non-standard generating function, denoted here $\mathcal S^{\rm out}_\Omega$, is more convenient. This function corresponds to the symplectic polar  coordinates $(r,\phi)$ in $\mathbb R^2$, when the origin lies inside $\Omega$. This function was introduced in \cite{bialy-outer}. 

$\mathcal S^{\rm out}_\Omega(\phi_1,\phi_2)$ is defined as follows. Given the values $\phi_1<\phi_2<\phi_1+\pi$, consider the segment with the ends lying on the rays with the angles $\phi_1$ and $\phi_2$, which is tangent to the curve $\gamma$ exactly at the middle. Then $\mathcal S^{\rm out}_\Omega(\phi_1,\phi_2)$ equals the area of the triangle bounded by the two rays and the segment.

\begin{figure}[h]\label{fig:generating}
	\centering
	\includegraphics[width=0.5\linewidth]{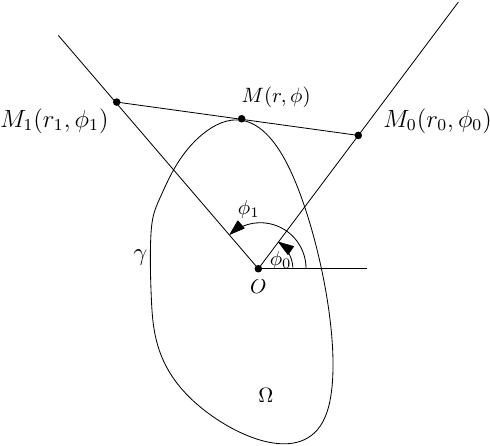}
	\caption{Generating function $\mathcal S^{\rm out}_\Omega$}
\end{figure}
The advantage of this generating function is in computing the action of a periodic orbit. 

Let $q\geq 3$ and consider a  $q$-periodic trajectory  for $B^{\rm out}_\Omega$ winding $p$ times about $\Omega$ for some  $\frac pq \in \Q\cap \left(0,\frac12\right)$, with $p,q$ coprime, that corresponds to the {rotation number} of the periodic orbit.  
Namely, for a  periodic orbit $\{M_k(r_k,\phi_k)\}_k$ of rotation number $\frac pq$, where $M_k :=  B^{\rm out}_\Omega( M_{k-1}), M_{q}=  M_0$, we can compute the action in terms of $\phi_k$, $ \phi_{q}=\phi_0+2\pi p$:  
$$
\mathcal A_{\frac pq}=\sum_{k=0}^{q-1} \mathcal S^{\rm out}_\Omega (\phi_k, \phi_{k+1}).
$$
It then follows that the {action of a periodic orbit} of rotation number $\frac pq$ equals the {area} of the corresponding circumscribed $q$-gon with (ordered) vertices $M_k$, $k=0, \ldots, q-1$. 
Then, if we denote by $\beta^{\rm out}_\Omega$ Mather's $\beta$-function associated to $B^{\rm out}_{\Omega}$,  we have
$$
\beta^{\rm out}_\Omega \left(\frac pq \right) =\frac 1q   
\mathcal A^{\rm out}_\Omega \left(\frac pq \right)\\
$$
where $\mathcal A^{\rm cir}_\Omega(\frac pq)$ denotes the maximal area of  $q$-gons with winding numbers $p$ circumscribed about $\Omega$.
\begin{example}
For the unit disk $\mathcal D$, $\beta^{\rm out}_{\mathcal D}(\rho)=\tan(\pi\rho)$. Indeed, choosing the centre as origin, one can check that the triangles with vertex $O$, $M_0(r_0,\phi_0)$ and $M(r,\phi)$ are always right and congruent. Thus, the length of the the segment $M_0(r_0,\phi_0)M_1(r_1,\phi_1)$ is constant along trajectories and equals to $2 \tan \left(\frac{\phi_1-\phi_0}2\right)$. 
\label{example2}
\end{example}

\medskip

\subsection{Symplectic billiards: Inner area billiards} \label{sec:symplbill}
{\it Symplectic billiards}  were  introduced by Albers and Tabachnikov in \cite{AlbersTabachnikov} to describe the evolution of a infinitesimally small ball inside $\Omega$, which
bounces on the boundary $\partial\Omega$ according to the following reflection law: given three successive impact points $p_1$, $p_2$ and $p_3$, the line joining $p_1p_3$ and the tangent line $T_{p_2}\partial\Omega$ of $\partial\Omega$ at $p_2$ are parallel (see Figure \ref{figsimplbill}). 
Unlike classical billiards, the reflection  law for symplectic billiards is not local. \\

\begin{figure}
	\centering
	\begin{tikzpicture}[line cap=round,line join=round,>=triangle 45,x=2.0cm,y=2.0cm]
		\clip(-2,-1.2) rectangle (2,1.35);
		\draw [rotate around={0:(0,0)}] (0,0) ellipse (2.82cm and 2cm);
		\draw [dash pattern=on 1pt off 1pt,domain=-3.19:3.33] plot(\x,{(--32-8.43*\x)/29.69});
		\draw [dash pattern=on 1pt off 1pt,domain=-3.19:3.33] plot(\x,{(-10.52-8.43*\x)/29.69});
		\draw [-latex] (1.07,-0.66) -- (0.72,0.37);
		\draw [-latex] (0.53,0.93) -- (-0.58,0.43);
		\draw (-0.58,0.43)-- (-1.41,0.05);
		\draw (0.72,0.37)-- (0.53,0.93);
		\begin{scriptsize}
			\draw[color=black] (1.15,-0.77) node {$p_1$};
			\draw[color=black] (0.56,1.02) node {$p_2$};
			\draw[color=black] (-1.5,-0.05) node {$p_3$};
			\draw[color=black] (1.5,0.8) node {$T_{p_2}\partial\Omega$};
			\draw[color=black] (0,0) node {$\Omega$};
		\end{scriptsize}
	\end{tikzpicture}
	\caption{Symplectic billiard} \label{figsimplbill}
\end{figure}
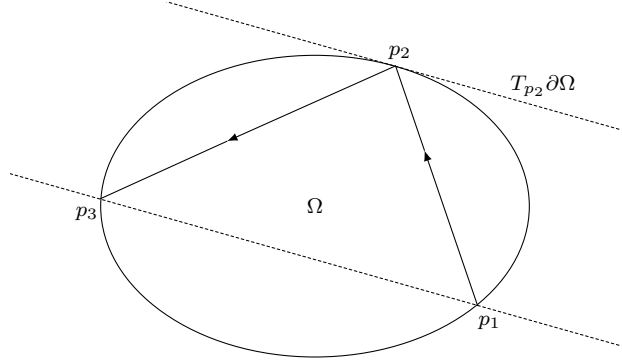

Let $\gamma: \R/ \Z \longrightarrow \R^2$ be a regular parametrization.
Given a  point $\gamma(t)$, denote by $\gamma(t^*)$  the other point on $\partial \Omega$ where the tangent line is parallel to that in $\gamma(t)$. 
According to the description above, the phase space of the symplectic billiard map in $\Omega$ is then the set of the oriented chords $\gamma(t_0) \gamma(t_1)$ where $t_0 < t_1 <t_0^*$,  according to the orientation of $\gamma$. Alternatively, the phase space can be described as the set 
$$ \mathcal X_\Omega:= \{(t_0, t_1): \; \omega(\dot \gamma(t_0), \dot \gamma (t_1)) >0\}$$ 
where $\omega$ denotes the standard area form in the plane ({\it i.e.}, the determinant of the matrix made by the two vectors). The vertical foliation consists of the chords with a fixed initial point. 
Then, the symplectic billiard map is given by:
\begin{eqnarray}\label{Bsymmap}
	B^{\tiny \rm symp}_\Omega: \mathcal X_\Omega &\longrightarrow &  \mathcal X_\Omega \nonumber\\
	(t_0,t_1) &\longmapsto& (t_1,t_2),
\end{eqnarray}
where $(t_1,t_2)$ is uniquely determined by the condition that   the tangent line to $\partial \Omega$ at $\gamma(t_1)$ is parallel to the line $\gamma(t_0)\gamma(t_2)$ (see Figure \ref{figsimplbill}). The map $B^{\tiny \rm sym}_\Omega$ is a  twist map and
it can be extended to the boundary of the phase space by continuity: $B^{\tiny \rm sym}_\Omega (t,t) := (t,t)$ and 
$B^{\tiny \rm symp}_\Omega (t,t^*) := (t^*,t)$.

\subsubsection{Generating function for symplectic billiards.}
\label{subsec:generating_function_symplectic}

Also in this case, $B^{\tiny \rm symp}_\Omega $ being a twist map, we can introduce its generating function (we refer to \cite{AlbersTabachnikov} for more details)
\begin{eqnarray*}
	\mathcal S^{\tiny \rm symp}_\Omega: \left\{ (t_0,t_1):\; t_0<t_1< t_0^*  \right\} &\longrightarrow& \R\\
	(t_0,t_1) &\longrightarrow& \mathcal S^{\tiny \rm symp}_\Omega(t_0,t_1)
\end{eqnarray*}
where $$\mathcal S^{\tiny \rm symp}_\Omega(t_0,t_1)=-\frac{1}{2}\omega (\gamma(t_0),\gamma(t_1)),$$ which is the negative of half of the area of the triangle with vertices $\gamma(t_0),\gamma(t_1)$ and $O$. Here again the minus sign is introduced in order to have a {positive} twist map.
Namely the twist condition holds:
$$
\partial^2_{12} \mathcal S^{\tiny \rm symp}_\Omega(t_0,t_1)=  - \omega(\dot \gamma(t_0), \dot \gamma(t_1))/2 <0.\\
$$

Observe that also in this case, as for outer billiards, $B^{\tiny \rm symp}_\Omega$ commutes with affine transformations of the plane.  Therefore, disks and ellipses are dynamically equivalent.

\medskip

Let us focus our attention on periodic orbits.

\noindent Let $q\geq 2$ and consider $\{(t_k,t_{k+1})\}_{k}$ a  $q$-periodic orbit  for $B^{\rm symp}_\Omega$ winding $p$ times about $\Omega$ for some  $\frac pq \in \Q\cap \left(0,\frac12\right)$, with $p,q$ coprime, that corresponds to the {rotation number} of the periodic orbit.  
We can compute the action 
$$
\mathcal{A}_{\frac pq} = 
\sum_{k=0}^{q-1} \mathcal S^{\rm symp}_\Omega (t_k, t_{k+1}).
$$
It follows   that the {action of a periodic orbit} of rotation number $\frac pq$ corresponds to the negative of the  {area} of the corresponding inscribed $q$-gon with (ordered) vertices $\gamma(t_k)$, $k=0, \ldots, q-1$. See also \cite{AlbersTabachnikov, AlbersTabachnikov_Dowker}.\\

If we denote by $\beta^{\rm symp}_\Omega$ Mather's $\beta$-function associated to $B^{\rm symp}_{\Omega}$,  we have
$$
\beta^{\rm symp}_\Omega \left(\frac pq \right) = -\frac 1q 
\mathcal A^{\rm ins}_\Omega \left(\frac pq \right),\\
$$
where  $\mathcal A^{\rm ins}_\Omega(\frac pq)$ denotes the maximal area of  $q$-gons with winding number $p$ inscribed in $\Omega$.
\begin{example}
For the unit disk $\mathcal D$, $\beta^{\rm symp}_{\mathcal D}(\rho)=-\frac 12 \sin(2\pi\rho)$.  Indeed, choosing the centre of $\mathcal D$ as the origin, and given three points of an orbit $p_1,p_2$ and $p_3$ we see that the chord $p_1p_3$ is orthogonal to the radius through $p_2$. Thus the triangles $p_1p_2O$ and $p_2p_3O$ are right and congruent, so their area is constant and equal to $\frac12 \sin \alpha$, where $\alpha$ stands for the angle between $p_1$, $O$ and $p_2$.
\label{example3}
\end{example}

\medskip

\subsection{Outer length billiards (or $4^{\rm th}$ billiards)} \label{sec:4thbill}
This model was recently introduced in \cite{AlbersTabachnikov_Dowker}.

Let us fix an arclength parametrization of $\gamma=\partial \Omega$ and let
$\ell$ denote the  length of $\gamma$.
This billiard map, denoted by $ B_{\Omega}^{4^{\rm th}}$ (since the term $4$-th billiard is a slang term for Outer length billiard) acts in the exterior of $\Omega$, according to the following rule (see Figure \ref{fig:fourth-rule}).
\begin{figure}[h]
	\centering
	\includegraphics[width=0.5\linewidth]{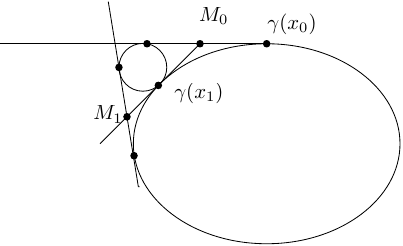}
	\caption{Outer length billiard rule}
	\label{fig:fourth-rule}
\end{figure}

On this picture, given a point $M_0$ in the exterior, consider two tangents through $M_0$ to $\gamma$ and the unique circle tangent to $\gamma$ at $\gamma(x_1)$ and to the tangent line at $\gamma(x_0)$. Then, the point 
$M_1=B_{\Omega}^{4^{\rm th}}(M_0)$ is defined as the intersection point of the tangent line at $\gamma(x_1)$ and the unique line, which is tangent to the circle and $\gamma$. We will denote the lengths $$\lambda_0:=|M_0-\gamma(x_0)|,\quad \lambda_1:=|M_0-\gamma(x_1)|.$$
One can prove \cite{AlbersTabachnikov_Dowker} that $ B_{\Omega}^{4^{\rm th}}$ is a {positive} twist map, and  $S_{\Omega}^{4^{\rm th}}(x_0,x_1)=\lambda_0+\lambda_1$ is a generating function for the outer length billiard.

Moreover, for any periodic configuration of rotation number $\frac{p}{q}\in (0, \frac 12) \cap \Q$, {\it i.e.},
$\{x_k\}_{k \in \mathbb{Z}}$, $x_{k+q}=x_k+p\ell$, we have that the action 
$$
\sum_{k=0}^{q-1}  S_{\Omega}^{4^{\rm th}}(x_k, x_{k+1}),
$$
equals the perimeter of the corresponding $q$-gon with (ordered) vertices $\{M_k\}_{k=0}^{q-1}$. Then, if we denote by $\beta_{\Omega}^{4^{\rm th}}$ Mather's $\beta$-function associated to $S_{\Omega}^{4^{\rm th}}$,  we have:
$$
\beta_{\Omega}^{4^{\rm th}} \left(\frac pq \right) = \frac 1q
\mathcal L_{\Omega}^{\rm circ}\left(\frac pq \right).\\
$$
where  $\mathcal L_{\Omega}^{\rm circ}$ denotes the minimal perimeter of  $q$-gons with winding numbers $p$ circumscribed about $\Omega$.
\begin{example}
	For the unit disk $\mathcal D$, $\beta_{\Omega}^{4^{\rm th}}(\rho)=2\tan\pi\rho$. Indeed, denote by $M$ the third vertex of the triangle circumscribing the smaller circle. Note that the triangle obtained connecting $\gamma(x_0)$, $O$ and $M$ is right and congruent to the one obtained connecting $M$, $O$ and the third tangency point in the picture. This means that the line connecting $M$ and $O$ bisects the angle at $M$ and at the same time is orthogonal to the tangent line through $\gamma(x_1)$. Thus, the length $M_0\gamma(x_1)$ is constant along orbits and equals $2 \tan (\alpha)$ where $\alpha$ is the angle between $\gamma(x_0)$, $\gamma(x_1)$ and $O$. 
	\label{example4}
\end{example}

\medskip

\section{Rigidity of Mather's $\beta$ function for billiards} \label{secrigiditybeta}

In this section we state our main results, namely, pointwise rigidity for the $\beta$-function of circular Birkhoff, symplectic and $4^{\rm th}$ billiards.
As noted in the Introduction, it is quite remarkable that the value of these functions at a single point is sufficient to determine whether the domain is a disk (or an ellipse, for symplectic billiards). 
All known results so far were {\it global}, in the sense that they required knowledge of the entire $\beta$ function or infinitely many of its values.\\

Before stating our main results, let us briefly recall these global results.\\

\subsection{Global rigidity of $\beta$ function for billiards}\label{sec:globalbeta}
Known results can be grouped into two different categories and reduce essentially to two strategies:
\begin{itemize}
\item Asymptotic expansion of the $\beta$ function at zero, seeking to identify what kind of information can be retrieved from the coefficients of this formal expansion
\item Differentiability properties of the $\beta$ function and their relation to integrability and to the Birkhoff conjecture.
\end{itemize}

\subsubsection{Asymptotic (formal) expansion at $0$}

In \cite{Siburg, SorDCDS}, properties of Mather's $\beta$  and explicit expressions for their (formal) Taylor expansions at $\omega=0$,  $\beta(\omega) \sim \sum_{k=0}^{\infty}  \beta_k \frac{\omega^k}{k!}$, have been obtained for Birkhoff billiards, and later extended to symplectic, outer and $4^{\rm th}$ billiards in \cite{BBN, baracco-bernardi-corentin}. The coefficients in these expressions involve the curvature of the boundary and its derivatives; for Birkhoff billiards, they are related to the so-called {\it Marvizi-Melrose invariants} \cite{MarviziMelrose}.

Although it seems quite a challenging task to recover the shape of the domain from the expressions of these coefficients (see for example \cite{BuKa}), some consequences can nevertheless be derived:

\begin{itemize}
\item For classical billiards (see \cite[Corollary 1]{SorDCDS} and \cite{Siburg}) the following inequality holds
$$
\beta_{3} + \pi^2 \beta_{1} \leq 0
$$
and equality holds if and only if $\Omega$ is a disk.
In particular, the above corollary says that if the first two coefficients $\beta_1$ and $\beta_3$ coincide with those of the $\beta$-function of a disk, then the domain must be a disk.

\item Similar results were later proved for other models of billiards (see \cite[Corollary 5.2]{BBN}). In particular, for outer and symplectic billiards coefficients $\beta_5$ and $\beta_7$ are sufficient to recognize whether the domain is an ellipse. One has  ($\lambda$ denotes the affine length of the boundary)
$$ 42\lambda^3 \beta^{\rm symp}_7 \leq 5! (\beta^{\rm symp}_5)^2
\qquad {\rm and} \qquad
7 \lambda^3 \beta^{\rm out}_7 \geq 170 (\beta^{\rm out}_5)^2
$$
with equality if and only if the domain is an ellipse.

\item For $4^{\rm th}$ billiards, it was proven in \cite[Corollary 7]{baracco-bernardi-corentin} that 
$$3 \beta^{4^{\rm th}}_3 + \pi^2 \beta_1^{4^{\rm th}} \leq 0$$
with equality if and only if it is a circle.

\end{itemize}

\medskip

\begin{remark}
In the case of Birkhoff and $4^{\rm th}$ billiards, a more challenging question is whether these expansions allow to recover that the domain is an  ellipse.
For Birkhoff billiards it is known that the the coefficients $\beta_1$ and $\beta_3$ determine univocally a given ellipse in the family of all ellipses (see \cite[Proposition 1]{SorDCDS}). Moreover, 
in \cite[Theorem 1]{Bialyellipses} an explicit expression of Mather's $\beta$-function for elliptic billiards has been computed in terms of elliptic integrals, which allows the author to prove that the value of $\beta$ at $\frac 14$ and its derivative at $0$ also determine univocally a given ellipse in the family of  ellipses. Similarly, one can prove that  the values of $\beta$ at $\frac 12$ and any other rational in $(0,\frac 12)$ are sufficient to determine a given ellipse amongst ellipses (see \cite[Theorem 6]{Bialyellipses}).

\end{remark}

  \subsubsection{Differentiability properties, integrability and Birkhoff conjecture}

An important result by  Mather \cite{Mather90} states that the function $\beta$ is differentiable at a rational point $\rho$ if and only if there exists an invariant curve consisting of periodic orbits with rotation number $\rho$. Moreover, all orbits lying on these invariant curves are action minimizing.

In particular, being $C^1$ on an interval implies that the twist map possesses invariant curves
for all rational rotation numbers in the interval of differentiability and, by a semi-continuity argument, one obtains invariant curves also for irrational rotation numbers. It is then possible to prove that these curves  foliate an open set of the phase space (a form of $C^0$-integrability). In the case of billiards, this observation allows one to translate many recent results on the Birkhoff conjecture in terms of rigidity properties of the $\beta$-functions.

\begin{itemize}
\item \textit{(Global differentiability)}. If $\beta$ is differentiable on the whole domain of definition $(0,1/2]$, then the corresponding billiard map is globally integrable. For Birkhoff billiards, it follows from \cite{Bialy} that the domain is  a disk. The result also holds for symplectic and outer billiards (in this case the domain can be also an ellipse); see, respectively, \cite{BaraccoBernardi} and \cite{bialy-outer}.

\item \textit{(Centrally symmetric case)}. It follows from the results in \cite{BialyMironov} (for classical billiards) and in \cite{BBN2} (for symplectic billiards) that if $\Omega$ is in addition centrally symmetric and the associated  $\beta$ function is differentiable on $(0,1/4]$, then it is an ellipse. 

\item \textit{(Perturbative case)}. It follows from the results in \cite{ADK, KS} (for classical billiards) and in \cite{Tsodikovich} (for symplectic billiards) that if $\Omega$ is a smooth strictly convex domain sufficiently $C^1$-close to an ellipse and its Mather's $\beta$-function is differentiable at all rationals $1/q$ with $q\geq 3$, then $\Omega$ must be an ellipse. \\
For classical billiards, this result could be actually generalized using \cite{Koval} (see also \cite{HKS}), considering integrability near the boundary. In particular, this would allow to consider $\beta$ only in a neighborhood of $0$.

\end{itemize}

\medskip

\subsection{Pointwise rigidity of $\beta$ for billiards: Main results}\label{sec:mainresults}

Let us state our main results.  \\

\begin{theorem}[{\bf Birkhoff billiards}] \label{thmBirkbill} \phantom{}	
\begin{itemize}
\item[{\bf (i)}] The following inequality holds:
	\begin{equation}\label{ineqbetabirkhoff}
	\beta_\Omega (\rho) \leq  \frac{|\partial \Omega|}{2\pi} \beta_{\D}(\rho) \qquad \forall\; \rho \in \Big[0,\frac 12\Big],
	\end{equation}
	where $|\partial \Omega|$ denotes the perimeter of $\Omega$ and $\D$  the unit disk.
\item[{\bf (ii)}]If equality   is achieved in \eqref{ineqbetabirkhoff} at some $\rho$, then the billiard map admits an invariant curve with constant angle of reflection equal to $ \pi \rho$  (it corresponds to a so-called  {\it Gutkin billiard}). In particular, if equality is achieved at $\frac 12$, then $\Omega$ is a constant width domain.
\item[{\bf (iii)}] There exist an explicit set $\mathcal R \supset \Q \cap \left(0,\frac{1}{2}\right)$, whose complement is countable and dense in $\left(0, \frac{1}{2}\right]$, such that if equality   is achieved  in \eqref{ineqbetabirkhoff} at some  $\rho \in \mathcal R$, then $\Omega$ is a disk.\\
\end{itemize}
\end{theorem}

\begin{remark}\label{remthmbirkhoff}
{\bf (i)} Notice that inequality (\ref{ineqbetabirkhoff}) for rotation number $\rho= \frac1n$, $n\geq 3$, implies \eqref{insperimeter} in Section \ref{seccomparison}, 
using that  $\beta_{\Omega}(\frac1n)=-\frac 1n \mathcal L^{\rm ins}\big(n)$  and Example \ref{example1}.\\
{\bf (ii)} 
Gutkin in \cite{Gutkin}, in relation to a floating problem for convex bodies, investigated the existence of regular, convex billiard tables that admit a caustic corresponding to billiard trajectories with a constant angle of reflection $\pi\delta$  which clearly include disks and constant-width domains. We call them {\it Gutkin's billiards}. Gutkin \cite{Gutkin} showed showed that if $\delta\in (0,\frac 12) \cap \Q $, then circular billiards are the only Gutkin's billiards. However, for irrational $\delta$ satisfying the equation $\tan (n\pi \delta) = n\, \tan (\pi \delta) $ for some $n\geq 2$, he constructed a real-analytic family of (dynamically non-equivalent) non-circular domains with such a property.  The collection of such $\delta$'s is a countable dense subset of $\left(0,\frac12\right)$ and corresponds to $(0,\frac 12) \setminus \mathcal R$. 
\end{remark}

\medskip

\begin{theorem}[{\bf Symplectic billiards}]	\label{thmsymplbill} \phantom{}
\begin{itemize}
\item[{\bf (i)}] The following inequality holds:
	\begin{equation}\label{ineqbetasympl}
	\beta^{\tiny \rm symp}_\Omega (\rho) \leq  \frac{|\Omega|}{\pi} \beta^{\tiny \rm symp}_{\D}(\rho) \qquad \forall\; \rho \in \Big[0,\frac 12\Big],
	\end{equation}
	where $|\Omega|$ denotes the area of $\Omega$ and $\D$ the unit disk. 
\item[{\bf (ii)}] If for some rotation number $\rho \in \left(0, \frac 12\right)$ equality is achieved in \eqref{ineqbetasympl}, then $\Omega$ is an ellipse.
\end{itemize}
\end{theorem}

\begin{remark}
Notice that inequality \eqref{ineqbetasympl} for rotation numer $\rho= \frac1n$, $n\geq 3$, implies \eqref{insarea} in Section \ref{seccomparison}, using that $
\beta^{\rm symp}_\Omega \left(\frac 1n \right) = - \frac 1n
\mathcal A^{\rm ins}_\Omega \left(n\right)$ and  Example \ref{example3}.
\end{remark}

\medskip

\begin{theorem}[{\bf $4^{\rm th}$ billiards}] \label{thm4thbill} \phantom{}
	\begin{itemize}
		\item[{\bf (i)}] The following inequality holds:
		\begin{equation}\label{ineqbeta4}
				\beta^{4^{\rm th}}_{\Omega} (\rho) \leq  \frac{|\partial {\Omega}|}{2\pi} \beta^{4^{\rm th}}_{\D}(\rho), \qquad \forall\; \rho\in \Big[0,\frac 12\Big),
		\end{equation}
		where $|\partial \Omega|$ denotes the perimeter of $\Omega$ and $\D$ the unit disk. 
		\item[{\bf (ii)}] If for some rotation number $\rho \in \left(0, \frac 12\right)$ equality is achieved in \eqref{ineqbeta4}, then $\Omega$ is an disk.
	\end{itemize}	
\end{theorem}

\medskip

\begin{remark}
Notice that inequality \eqref{ineqbetasympl} for rotation numer $\rho= \frac1n$, $n\geq 3$, implies \eqref{circperimeter} in Section \ref{seccomparison}, using that 
$\beta^{4^{\rm th}}_\Omega \left(\frac 1n \right) =  \frac 1n \mathcal L^{\rm circ}_{\Omega} \left(n\right)$ and  Example \ref{example4}.
\end{remark}

\bigskip

\subsection{Proof of Theorem \ref{thmBirkbill} (Birkhoff billiards)}

\begin{proof} 
{\bf (i)} Fix $\frac{p}{q} \in\left(0,\frac 12\right]\cap \mathbb{Q}$, where $p$ and $q$ are coprime positive integers. For every $\varphi \in \mathbb{S}^1$, consider a periodic configuration of rotation number $\frac{p}{q}$ starting at $\varphi$:
\[
\{\phi_k(\varphi)\}_{k\in \mathbb{Z}}, \quad \text{where} \quad \phi_k(\varphi) := \varphi + 2\pi \frac{p}{q} k.
\]
Its action is given by:
\begin{align*}
\mathcal{A}_{\frac{p}{q}}(\varphi) &= \sum_{k=0}^{q-1} S(\phi_k, \phi_{k+1}) = -2 \sum_{k=0}^{q-1} h\left( \frac{\phi_{k+1} + \phi_k}{2} \right) \sin\left( \frac{\phi_{k+1} - \phi_k}{2} \right) \\
&= -2 \sum_{k=0}^{q-1} h\left( \varphi +  (2k+1)\frac{ \pi p}{q} \right) \sin\left(  \frac{\pi p}{q} \right).
\end{align*}
Therefore, we have the inequality for every $\varphi$:$$ \frac 1q \mathcal A_{\frac{p}{q}}(\varphi)=
 -\frac 2q \sum_{k=0}^{q-1} h\left( \varphi +  (2k+1)\frac{ \pi p}{q}\right) \sin\left(  \frac{\pi p}{q} \right)\geq\beta_\Omega\left( \frac{p}{q} \right).
$$
Integrating with respect to $\varphi$, we have:
\begin{align*}
 \frac 1q\int_0^{2\pi} \mathcal{A}_{\frac{p}{q}}(\varphi) \,d\varphi &=-\frac{2}{q}\sin\left(  \frac{\pi p}{q} \right) \sum_{k=0}^{p-1} \int_0^{2\pi} h\left( \varphi + (2k+1)\frac{ \pi p}{q}\right) \,d\varphi \\
	&=-2  \sin\left(  \frac{\pi p}{q} \right) \int_0^{2\pi} h(\varphi) \,d\varphi \\
	&= -2\sin\left(  \frac{\pi p}{q} \right) |\partial\Omega|\geq 2\pi\beta_\Omega\left( \frac{p}{q} \right),
\end{align*}
where $|\partial\Omega|$ denotes the length of the boundary of $\Omega$ and 
we used the fact that $\int_0^{2\pi} h(\varphi) \,d\varphi = |\partial\Omega|$ (see for instance \cite{support_function}).
Using  $\beta_{\mathcal D}\left(\frac pq\right) = - 2\sin(\frac{\pi p}{q})$ (see Example \ref{example1}) the last inequality reads:

$$\beta_{\mathcal D}\left(\frac pq\right) |\partial\Omega|\geq 2\pi\beta_\Omega\left( \frac{p}{q} \right).
$$

Now,using the fact that $\beta$-function is continuous and the density of rationals, we can extend this inequality to all rotation numbers:
$$
\beta_\Omega (\rho) \leq \frac{|\partial\Omega|}{2 \pi} \beta_{\D}(\rho) \qquad \forall\; \rho\in \Big[0,\frac 12\Big].
$$
which is exactly the inequality we aim to prove.\\

\noindent {\bf (ii) \& (iii)} Assume now that equality holds at some $\rho \in (0,1/2]$ (clearly the result holds for $\rho=0$).
It follows (proceeding as in the proof of Theorem \ref{thm1} and Lemma \ref{linearconfigurations}) that every $2\pi\rho$-equispaced  configuration $\phi_k:=\varphi + 2\pi k \rho$, $k\in \Z$, corresponds to an orbit of $B_\Omega$ for every $\varphi\in [0,2\pi)$.

In particular, the corresponding billiard trajectory has a constant angle of reflection:
\[
\frac{\phi_{k+1} - \phi_k}{2} \equiv \pi \rho.
\]
Hence, it is a Gutkin billiard (see \cite{Gutkin} and Remark \ref{remthmbirkhoff} (ii)).
This means that:
$$
\partial_1 S(\phi_k, \phi_{k+1}) + \partial_2 S(\phi_{k-1}, \phi_k) = 0 \qquad \forall\; k\in \Z\quad {\rm and} \quad \forall \varphi\in [0,2\pi),
$$
which implies (see (\ref{generatingfunction-birkhoff})) that
$$
\left[h'(\varphi + 2\pi \rho) + h'(\varphi)\right] \,\sin (\pi \rho) -
\left[h(\varphi + 2\pi \rho) - h(\varphi)\right] \,\cos (\pi \rho) = 0\qquad \forall\; \varphi\in [0,2\pi).
$$

\medskip
If we consider the Fourier series of $h$, $h(\varphi):= \sum_{n\in \Z} c_n e^{in \varphi}$, the above equality becomes:

\begin{eqnarray*}
&& \sum_{n\in \Z} \left[ i  n\,c_n\, (e^{2\pi i n \rho} + 1) \,\sin (\pi \rho)  - 
c_n\, (e^{2\pi i n \rho} - 1) \,\cos (\pi \rho)
\right]\, e^{i n \varphi} \; =0\\
&\Longleftrightarrow& 
2i \sum_{n\in \Z} \underbrace{\left( 
n\, \cos (n\pi\rho) \sin (\pi\rho) - \sin(n\pi\rho) \cos (\pi\rho)
\right)}_{:= k_{\rho}(n)}\, c_n e^{in\pi\rho} e^{in\varphi}
=0.\\
\end{eqnarray*}
Hence, we can conclude that $c_n=0$ for all $n\in \Z$ such that  $k_{\rho}(n) \neq 0$.\\

For for $\rho=\frac 12$, we have $k_{\frac{1}{2}}(2n+1)=0$ for every $n \in \Z$. This and the previous observation imply  that support function  is of the form
$$
\qquad h(\varphi):= \sum_{n\in \Z} c_{2n} e^{i\, 2n \varphi} \qquad {\Longleftrightarrow} \qquad h(\varphi + \pi) + h(\varphi) \equiv 2c_0 \quad \forall\; \varphi \in [0,2\pi).
$$
This characterizes the support functions of constant width domains (see \cite{Gutkin, Taba}).\\

For $\rho \neq \frac 12$, 
\begin{equation}\label{tangentequation}
k_{\rho}(n) = 0 \qquad \Longleftrightarrow \qquad
\tan (n\pi \rho) = n\, \tan (\pi \rho).
\end{equation}
Trivially, $k_{\rho}(0) = k_{\rho}(\pm 1) = 0$. 
Let us define the set of $\rho$ that do not admit other non-trivial solutions:
\begin{equation}\label{defR}
\mathcal R := \Big\{\rho \in \Big(0, \frac 12\Big):\;  \tan (n\pi \rho) \neq n\, \tan (\pi \rho) \quad \forall\, n\in \Z, \; |n| \geq 2 \Big\}.
\end{equation}
Then, if $\rho \in \mathcal R$, we conclude:
$$
c_n = 0 \quad \forall\; n\neq 0, \pm 1 \qquad  \Longrightarrow \qquad h(\varphi) = c_0 + c_1 e^{i\varphi} + \overline{c_1} e^{-i\varphi} = c_0 +  2\alpha \cos (\varphi + \beta)
$$
where $c_1 = {\alpha} e^{i \beta}$, $\alpha \geq 0$. This corresponds to the support function of a disk of radius $c_0$ and center in $(2\alpha \cos \beta, - 2\alpha \sin \beta)$.\\

It was proven in \cite[Theorem 1]{Cyr} (see also \cite{FlorianPrachar}) that $\mathcal R \supset \big(0, \frac 12\big)\cap \Q$. Moreover,  
$\big(0, \frac 12\big) \setminus \mathcal R$ is countable and dense in $\big(0, \frac 12\big)$ (see \cite{Gutkin}).
\end{proof}

\medskip

\subsection{Proof of Theorem \ref{thmsymplbill} (Symplectic billiards)}

We use  a clever parametrization of the boundary curve $\gamma$ used by Sas in \cite{Sas}.

\begin{proof}
Consider a coordinate system whose \( x \)-axis coincides with the largest chord (or one of the largest chords) of $\Omega$ and whose origin lies at the midpoint of this chord. Moreover, by rescaling, let us assume that the length of the largest chord is $2$.

The boundary of $\Omega$ can  be represented by the equations:
\begin{equation}\label{defgamma}
\gamma(t) := (\cos t,  e(t) \sin t) \qquad t\in [0,2\pi),
\end{equation}
where \( e(t) \), except at the points \( t = 0 \) and \( t = \pi \), is a  {$C^2$}, 
positive, \( 2\pi \)-periodic function 
in \( t \). Furthermore:
\[
\lim_{s \to 0} e(s) s = 0, \quad \lim_{s \to 0} e(\pi - s) s = 0.
\]

Let $\frac pq \in \Q\cap \left(0,\frac12\right)$, $p,q$ coprime, and  consider  the inscribed $q$-gon with winding number $p$ with (ordered) vertices corresponding to the configurations: 
\[
t_0=t, \quad t_1 = t + \frac{2\pi p}{q}, \quad \ldots, \quad t_{q-1} = t + (q-1) \frac{2\pi p}{q}.
\]

The action of this configuration equals the negative of the area of the corresponding $q$-gon (we use that \(t_{q} = t_0 + 2\pi p\)):
\begin{align*}
\mathcal A_{\frac pq} (t) :&= -\frac 12 \sum_{k=0}^{q-1} \omega (\gamma (t_i), \gamma (t_{i+1})) \\
& =-\frac{1}{2} \sum_{i=0}^{q-1} e(t_i) \sin t_i (\cos t_{i-1} - \cos t_{i+1})\\
& =-\frac{1}{2} \sin\left(\frac{2\pi p}{q}\right)\sum_{i=0}^{q-1} e\left(t + \frac{2\pi p}{q}\right) \sin^2\left(t+\frac{2\pi p}{q}\right)\geq  \beta_\Omega^{\rm ins}\bigg(\frac pq \bigg).\\
\end{align*}

Integrating last inequality with respect to $t$ and using Gauss-Green formula, we obtain:
\begin{align*}
\frac{1}{q} \int_0^{2\pi} \mathcal A_{\frac pq} (t) \, dt &= - \sin \left(\frac{2\pi p}{q}\right) \int_0^{2\pi} e(t) \sin^2 t \, dt = -  \sin \left(\frac{2\pi p}{q}\right)\, |\Omega|\\&\geq 2\pi \beta_\Omega^{\rm ins}\bigg(\frac pq \bigg).
\end{align*}

Taking into account $\beta^{\rm symp}_\D=-\frac 12\sin(2\pi\rho)$ (see Example \ref{example3}) we get exactly (\ref{ineqbetasympl}) for rational rotation numbers. Hence, by continuity, {\bf (i)} holds true for all rotation numbers.

To  prove {\bf (ii)}
let us assume that that there exists $\rho \in (0,\frac 12)$ such that there is equality in (\ref{ineqbetasympl}) (the result clearly holds for $\rho=0, \frac 12$). 
This implies, exactly as in the  proofs of Lemma \ref{linearconfigurations} (see also Theorems \ref{thm1} and \ref{thmBirkbill}), that for every $t \in [0,\pi)$ the configuration
$$ \{t_k\}_{k\in\Z}, \quad t_{k} = t + k2\pi\rho $$
corresponds to an orbit of the symplectic billiard in $\Omega$. Hence,
$
\gamma'(t) $ is parallel to  $\gamma(t+ 2\pi \rho) - \gamma(t-2\pi \rho)$ for every $t\in [0,2\pi)$, namely:
$$
\gamma'(t)= \alpha(t) \left( \gamma\left(t+ \frac{2\pi p}{q}\right) - \gamma\left(t- \frac{2\pi p}{q}\right)\right) \qquad \forall\; t\in [0,2\pi),
$$
where $\alpha : [0,2\pi)\longrightarrow \R$.
Recalling the specific form of the parametrization in \eqref{defgamma} and looking at the $x$-component, the above equality reads:
\begin{eqnarray*}
- \sin t = \alpha(t) \left(\cos \left(t+ 2\pi \rho)\right) - \cos \left(t- 2\pi \rho\right) \right) = -2 \alpha(t)\, \sin t \,\sin\left( 2\pi \rho\right),
\end{eqnarray*}
from which it follows that $\alpha(t) = \frac{1}{2 \sin( 2\pi \rho)}$.

Looking now at the $y$-component of the above relation, we obtain that:
\[
y'(t) = \frac{y(t + 2\pi \rho) - y(t - 2\pi \rho)}{2 \sin (2\pi \rho)} \qquad \forall\; t\in [0,2\pi).
\]

Expanding $y(t)$ in Fourier series, $y(t) = \sum_{n\in \Z}^{\infty} c_n e^{it}$, the previous equality becomes
\begin{align*}
 \sum_{n\in \Z} i n c_n e^{int} \;&=\; \frac{1}{{2 \sin(2\pi \rho)}} \sum_{n\in \Z} c_n \left(e^{i 2\pi n \rho} - e^{-i 2\pi n \rho}\right) e^{int}\\
   0&=\sum_{n\in \Z} c_n \underbrace{\left(n -  \frac{\sin \left(n 2\pi \rho\right)}{\sin \left(2\pi \rho\right)}\right)}_{:= w_{\rho}(n)} e^{int}.
\end{align*}

Hence, we can conclude that $c_n=0$ for all $n\in \Z$ such that  $w_{\rho}(n) \neq 0$.
Observe that:
\begin{equation}\label{tangentequation}
w_{\rho}(n) = 0 \qquad \Longleftrightarrow \qquad
\sin \left(n 2\pi  \rho\right) = n\, \sin \left(2\pi  \rho\right).
\end{equation}

\medskip

Clearly, there are trivial solutions to the above equation:
$$
w_{\rho}(0) = w_{\rho}(\pm 1) = 0.
$$

It can be shown, see e.g. \cite[Lemma 4.3]{bialy-tsodikovich} that these are the only solutions.
\
Therefore,
\[
y(t) = e(t) \sin t = c_0 +  c_1 e^{it} + \overline{c_1} e^{-it}.
\]
Since  \( y(0) = y(\pi) = 0 \), then
\begin{eqnarray*}
&& c_0 + c_1 + \overline{c_1} =  c_0 + 2\,{\rm Re}\,(c_1) = 0 \\
&& c_0 - c_1 - \overline{c_1} =  c_0 - 2\,{\rm Re}\,(c_1) = 0.
\end{eqnarray*}
Therefore $c_0 = {\rm Re}\,(c_1) = 0$ and, if we denote $c_1= i\alpha$ with $\alpha\in \R$:
$$
y(t) = e(t) \sin t =  i \alpha \,e^{it} -  i\alpha \,e^{-it} =  - 2 \,\alpha\, \sin t,
$$
and we conclude  that \( e(t)  \) is identically constant, which means the curve is an ellipse.

\end{proof}

\medskip

\subsection{Proof of Theorem \ref{thm4thbill} ($4^{\rm th}$ billiards)}

For $\gamma=\partial\Omega$ we will use envelope coordinates to compute 
the generating function $S_{\Omega4}=\lambda_0+\lambda_1$.
\begin{figure}[h]\label{lambda0lambda1}
	\centering
	\includegraphics[width=0.4\linewidth]{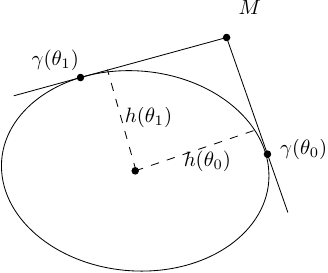}
	\caption{Computing $\lambda_0, \lambda_1$ in envelope coordinates}
\end{figure}

Let us denote by $h$ the support function of $\Omega$. We have the following formulae.
\begin{lemma}\label{lemma:lambdas}
$$
\lambda_0=-h'(\theta_0)+\frac{h(\theta_1)}{\sin(\theta_1-\theta_0)}-h(\theta_0)\cot(\theta_1-\theta_0),
$$
$$
\lambda_1=h'(\theta_1)+\frac{h(\theta_0)}{\sin(\theta_1-\theta_0)}-h(\theta_1)\cot(\theta_1-\theta_0).
$$
Therefore,
$$
S^{4^{\rm th}}_{\Omega}(\theta_0,\theta_1)=\lambda_0+\lambda_1=h'(\theta_1)-h'(\theta_0)+\bigg(h(\theta_1)+h(\theta_0)\bigg)\tan(\frac{\theta_1-\theta_0}{2}).
$$
\end{lemma}

\begin{proof}
	Coordinates of the points $\gamma(\theta_0)$ and $\gamma(\theta_1)$ can be found using formula (\ref{eq:envelope}) yielding 
	\begin{align*}
		\gamma (\theta_0) &= h(\theta_0)\, (\cos \theta_0, \sin \theta_0) + h'(\theta_0)\, (- \sin \theta_0, \cos \theta_0),\\
		\gamma (\theta_1) &= h(\theta_1)\, (\cos \theta_1, \sin \theta_1) + h'(\theta_1)\, (- \sin \theta_1, \cos \theta_1).
	\end{align*}
	
	Next, one needs to find the coordinates of the point $M$ as the intersection of the tangent lines. Thus  $(x_M,y_M)$, the coordinates of $M$, satisfy the system:
	\[
	\begin{cases}
			\cos\theta_0 x+\sin\theta_0 y=h(\theta_0) \\
		\cos\theta_1 x+\sin\theta_1 y=h(\theta_1).
	\end{cases}
	\]

	Solving this system we get:
	$$
	(x_M,y_M)=\frac{1}{\sin(\theta_1-\theta_0)}\bigg(h(\theta_0)\sin\theta_1-
	h(\theta_1)\sin\theta_0,\ h(\theta_1)\cos\theta_0-
	h(\theta_0)\cos\theta_1\bigg).
	$$Hence we get:
	
\begin{eqnarray*}
\lambda_0 &=& \frac{x_M-x_{\gamma(\theta_0)}}{\cos(\theta_0+\pi/2)} \\
&=& - \frac{h(\theta_0)\sin\theta_1-
		h(\theta_1)\sin\theta_0-\sin(\theta_1-\theta_0)(h(\theta_0) \cos \theta_0- h'(\theta_0)\sin \theta_0)}{\sin\theta_0{\sin(\theta_1-\theta_0)}}.
\end{eqnarray*}
Substituting into the numerator of this formula the identity
$$\sin\theta_1=\sin\theta_0\cos(\theta_1-\theta_0)+\cos\theta_0\sin(\theta_1-\theta_0),$$ 
we get precisely the first claim of  Lemma \ref {lemma:lambdas}.
Analogously, the second formula of Lemma \ref{lemma:lambdas} is proved.
	\end{proof}
	
	\medskip

\begin{proof} ({\it Theorem \ref{thm4thbill}}) {\bf (i)} Let us prove the inequality for any $\rho=\frac{p}{q} \in (0, \frac 12)\cap \Q$, the irrational case will follow by continuity. Consider the equispaced configurations with rotation number $\rho$ which starts with an arbitrary $\vartheta$
	\[
	\theta_i=\vartheta+k\frac{2\pi p}{q}\qquad  k=0,...,q-1
	\]
	and $\theta_q=\vartheta+2\pi p$.
We compute the action of this configuration using Lemma \ref{lemma:lambdas}:
$$\mathcal{A}_{\frac pq}(\vartheta)= \sum_{k=0}^{q-1} S^{4^{\rm th}}_{\Omega4}(\theta_k, \theta_{k+1}) =2\sum_{k=0}^{q-1}
h\bigg(\vartheta+k\frac{2\pi p}{q}\bigg)\tan\frac{\pi p}{q}.
$$
By the properties of the $\beta$-function we have the inequality for every $\vartheta$:
$$
\frac 1q\mathcal{A}_{\frac pq}(\vartheta)=\frac 2q \sum_{k=0}^{q-1}
h\bigg(\vartheta+k\frac{2\pi p}{q}\bigg)\tan\frac{\pi p}{q}\geq \beta^{4^{\rm th}}_{\Omega}(\frac pq).
$$
Integrating last inequality with respect to $\vartheta$ and using $\int_{0}^{2\pi}h(\vartheta)d\vartheta=|\partial\Omega|$ we get
$$
2|\partial\Omega|\ \tan (\frac{\pi p}q)\geq2\pi \beta^{4^{\rm th}}_{ \Omega}\bigg(\frac pq\bigg)
$$ Using in this inequality the expression for $\mathcal{D}$, 
$$
\beta^{4^{\rm th}}_{\mathcal D}\bigg(\frac pq\bigg)=2\tan\bigg(\frac {\pi p}{q}\bigg),
$$ we get the inequality (\ref{ineqbeta4}).

 {\bf (ii)} Let us assume now that for some $\rho\in \left(0, \frac12\right)$ there is an equality in (\ref{ineqbeta4}).
It then follows from the proof and Lemma \ref{linearconfigurations} that for any $\vartheta$ the equispaced configuration 
$\{\theta_k=\vartheta+2\pi \rho\}_k$  has a minimal average action and therefore is an orbit of the $4^{\rm th}$billiard.
Since the orbits are equispaced, the triangle with vertices $M_0,M_1$ and $P$ in Figure \ref{fig:proof4} is isosceles. Since we are considering an orbit, the bisector of the angle at $P$ must be also orthogonal to the side connecting $M_0$ to $M_1$. Thus, their intersection is the point $\gamma(x_1)$ which is the middle point of $[M_0,M_1]$ as well. So, we get the following equation:
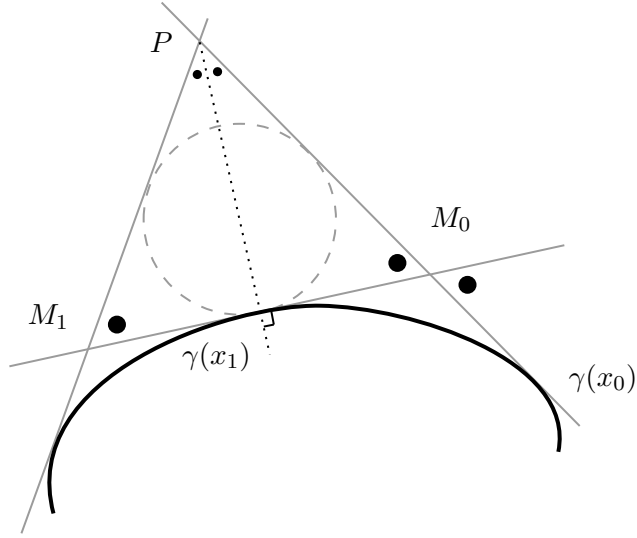
\begin{figure}
	\centering
	\label{fig:proof4}	
	\tikzset{every picture/.style={line width=0.75pt}} 
	
	\begin{tikzpicture}[x=0.75pt,y=0.75pt,yscale=-1,xscale=1]

		\draw [color={rgb, 255:red, 155; green, 155; blue, 155 }  ,draw opacity=1 ]   (324.33,3) -- (533,216) ;
		\draw [color={rgb, 255:red, 155; green, 155; blue, 155 }  ,draw opacity=1 ]   (248.33,186) -- (525.33,125.08) ;
		\draw [color={rgb, 255:red, 155; green, 155; blue, 155 }  ,draw opacity=1 ]   (347.33,11) -- (254.33,270) ;
		\draw  [color={rgb, 255:red, 155; green, 155; blue, 155 }  ,draw opacity=1 ][dash pattern={on 4.5pt off 4.5pt}] (315.33,112) .. controls (315.33,85.49) and (336.82,64) .. (363.33,64) .. controls (389.84,64) and (411.33,85.49) .. (411.33,112) .. controls (411.33,138.51) and (389.84,160) .. (363.33,160) .. controls (336.82,160) and (315.33,138.51) .. (315.33,112) -- cycle ;
		\draw  [dash pattern={on 0.84pt off 2.51pt}]  (343.33,23) -- (378.33,180.08) ;
		\draw  [fill={rgb, 255:red, 0; green, 0; blue, 0 }  ,fill opacity=1 ] (298,165) .. controls (298,162.79) and (299.79,161) .. (302,161) .. controls (304.21,161) and (306,162.79) .. (306,165) .. controls (306,167.21) and (304.21,169) .. (302,169) .. controls (299.79,169) and (298,167.21) .. (298,165) -- cycle ;
		\draw  [fill={rgb, 255:red, 0; green, 0; blue, 0 }  ,fill opacity=1 ] (438,134) .. controls (438,131.79) and (439.79,130) .. (442,130) .. controls (444.21,130) and (446,131.79) .. (446,134) .. controls (446,136.21) and (444.21,138) .. (442,138) .. controls (439.79,138) and (438,136.21) .. (438,134) -- cycle ;
		\draw    (379,158) -- (380.33,165) ;
		\draw    (375.33,166) -- (380.33,165) ;
		\draw [line width=1.5]    (270.33,260) .. controls (252.33,191) and (357.33,154.75) .. (403.33,155.5) .. controls (449.33,156.25) and (531.33,182.75) .. (522.33,229) ;
		\draw  [fill={rgb, 255:red, 0; green, 0; blue, 0 }  ,fill opacity=1 ] (473,145) .. controls (473,142.79) and (474.79,141) .. (477,141) .. controls (479.21,141) and (481,142.79) .. (481,145) .. controls (481,147.21) and (479.21,149) .. (477,149) .. controls (474.79,149) and (473,147.21) .. (473,145) -- cycle ;
		\draw  [fill={rgb, 255:red, 0; green, 0; blue, 0 }  ,fill opacity=1 ] (340.33,39.21) .. controls (340.33,38.22) and (341.14,37.42) .. (342.12,37.42) .. controls (343.11,37.42) and (343.92,38.22) .. (343.92,39.21) .. controls (343.92,40.2) and (343.11,41) .. (342.12,41) .. controls (341.14,41) and (340.33,40.2) .. (340.33,39.21) -- cycle ;
		\draw  [fill={rgb, 255:red, 0; green, 0; blue, 0 }  ,fill opacity=1 ] (350.74,38.8) .. controls (350.18,37.99) and (350.4,36.87) .. (351.21,36.31) .. controls (352.03,35.76) and (353.15,35.97) .. (353.7,36.79) .. controls (354.26,37.6) and (354.05,38.72) .. (353.23,39.28) .. controls (352.41,39.83) and (351.3,39.62) .. (350.74,38.8) -- cycle ;

		\draw (526,183.4) node [anchor=north west][inner sep=0.75pt]    {$\gamma ( x_{0})$};
		\draw (333,172.4) node [anchor=north west][inner sep=0.75pt]    {$\gamma ( x_{1})$};
		\draw (256,153.4) node [anchor=north west][inner sep=0.75pt]    {$M_{1}$};
		\draw (457,104.4) node [anchor=north west][inner sep=0.75pt]    {$M_{0}$};
		\draw (317,16.4) node [anchor=north west][inner sep=0.75pt]    {$P$};
			
	\end{tikzpicture}
	
	\caption{A portion of an equispaced $4^{\rm th}$ billiard trajectory}
\end{figure}
$$
\lambda_0(M_1)=\lambda_1(M_0),
$$which, due to the formulae of Lemma \ref{lemma:lambdas}, is equivalent to
\begin{equation}
-h'(\vartheta+\delta)+\frac{h(\vartheta+2\delta)}{\sin\delta}-h(\vartheta+\delta)\cot\delta=h'(\vartheta+\delta)+\frac{h(\vartheta)}{\sin\delta}-h(\vartheta+\delta)\cot\delta,
\end{equation}where $\delta=2\pi\rho$.
Thus 
\begin{equation}
2h'(\vartheta+\delta)\sin\delta=h(\vartheta+2\delta)-h(\vartheta).
\end{equation}
Considering the Fourier series of $h$, $h(\varphi):= \sum_{n\in \Z} c_n e^{in \varphi}$, the above equality becomes:
$$
2c_n(in e^{in\delta}\sin\delta)=c_n(e^{2in\delta}-1).
$$
Therefore
$$
c_n(2in\sin\delta)=c_n(2i\sin n\delta).
$$
The equation $\sin n\delta=n\sin\delta$ is obviously satisfied for $n=0,\ \pm 1$, but has no other  solution for $|n|>1$ (see for example \cite[Lemma 4.3]{bialy-tsodikovich}). 
This implies that $c_n=0$ for all $ |n|>1$, meaning (see the proof of Theorem \ref{thmBirkbill})  that $\gamma$ is a circle.
\end{proof}

\bigskip
\section{Outer billiards: Examples and Counterexamples }\label{sec:counterexouterbill}

The following natural question arises:

\medskip

\noindent{\bf Question:} {\it Given an Outer billiard in a  domain $\Omega$ (with the same hypothesis as above), can one estimate $\beta^{\tiny \rm out}_\Omega$  similarly  to what we did for Birkhoff, symplectic billiards and $4^{\rm th}$ billiards? Namely, does the following inequality hold true
$$	\beta^{\tiny \rm out}_\Omega (\rho) \leq  \frac{|\Omega|}{\pi} \beta^{\tiny \rm out}_{\D_1}(\rho) \qquad \forall\; \rho\in \Big[0,\frac 12\Big)\; ?$$
Equivalently,  can one estimate the minimal area of circumscribed polygons about $\Omega$ by the area of regular polygons circumscribing a disk of the same area as $\Omega$?}\\

We will see in the following two subsections that the answer is negative at least for  rotation numbers  $1/3$ and $1/4$. This will be achieved by establishing some relations between the dynamics of symplectic and outer billiards.\\

In the sequel we use the generating functions $\mathcal S_\Omega^{\rm out}$ and $\mathcal S^{\rm symp}_\Omega$ introduced in Sections \ref{subsec:generating_function_symplectic} and \ref{subsec:generating_function_outer},   as well as the interpretation of the action of periodic orbit as the area of corresponding circumscribed, respectively inscribed, polygon.\\

\subsection{Rotation number $1/3$} \label{sec:ex13}

Let us start with the following observation, which is an immediate consequence of Thale's intercept theorem (see Figure \ref{figtriangle}). 

\begin{lemma}\label{lemmarelationactions}
Let  $(A,B,C)$ be  vertices of a circumscribed triangle and $(a,b,c)$  the tangency points.  Then $(A,B,C)$ is a 3-periodic orbit    for $B^{\tiny \rm out}_\Omega$ if and only if, the points 
$(a,b,c)$  form  a 3-periodic orbit of $B^{\tiny \rm symp}_\Omega$ (see Figure \ref{figtriangle}).
Moreover, the areas of these triangles (hence, the actions of the corresponding) are related in the following way:
$${\rm Area} \ (\triangle ABC)=4 \, {\rm Area}(\triangle  abc)$$
\end{lemma}

\begin{figure}[h]
	\includegraphics[width=0.3\linewidth]{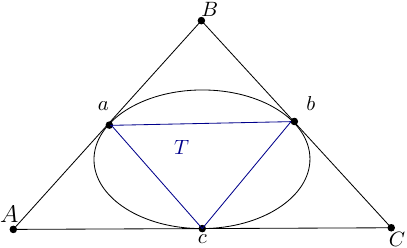}
	\caption{Outer billiard orbit $ABC$ and corresponding Symplectic billiard orbit $abc$}
	\label{figtriangle}
\end{figure}

\medskip
\begin{remark} \label{reminvcurvesperiod3}
It follows easily from the previous Lemma that $B^{\tiny \rm out}_\Omega$ admits an invariant curve consisting of periodic orbits of period $3$ if and only if $B^{\tiny \rm symp}_\Omega$ does.
Non-trivial functional family of domains $\Omega$ for which $B^{\tiny \rm out}_\Omega$ and $B^{\tiny \rm symp}_\Omega$ admit invariant curves consisting of periodic points of period $3$ were constructed by Genin-Tabachnikov \cite{genin-tabachnikov}.
\end{remark}

\medskip

\begin{corollary} \label{propouter13}
The following inequality holds:
\begin{equation}\label{beta1/3}
	\beta^{\tiny \rm out}_\Omega \left(\frac{1}{3}\right) + 4\,\beta^{\tiny \rm symp}_\Omega \left(\frac{1}{3}\right) \leq 0.
\end{equation}
Moreover, equality holds if and only if $B^{\tiny \rm out}_\Omega$ (or equivalently $B^{\tiny \rm symp}_\Omega$) admits an invariant curve consisting of periodic points of period $3$.
\end{corollary}

\medskip

\begin{proof}
We have that
$$
{\rm Area} \ (\triangle ABC)=4 \, {\rm Area}(\triangle  abc).
$$
It follows from the property of  of $\beta$-function:
$$
\beta^{\tiny \rm out}_\Omega \left(\frac{1}{3}\right) \leq {\rm Area} \ (\triangle ABC) \quad {\rm and}\ 
\beta^{\tiny \rm symp}_\Omega \left(\frac{1}{3}\right) \leq - {\rm Area}(\triangle  abc),
$$
which proves (\ref{beta1/3}).

If $B^{\tiny \rm out}_\Omega$ (and therefore $B^{\tiny \rm sym}_\Omega$) admits an invariant curve consisting of periodic points of period $3$ (see Remark \ref{reminvcurvesperiod3}), then we have the equalities:

$$
\beta^{\tiny \rm out}_\Omega \left(\frac{1}{3}\right) = {\rm Area} \ (\triangle ABC)  \quad {\rm and}\ 
\beta^{\tiny \rm symp}_\Omega \left(\frac{1}{3}\right) = - {\rm Area} \ (\triangle abc)),
$$ since all orbits on the curve are action-minimizing and hence there is  equality also in (\ref{beta1/3}). 

{In other direction, assume equality \eqref{beta1/3} holds. Then, if there were a non minimal orbit $A'B'C'$ for the outer billiard we could find an orbit of the symplectic billiard yielding (we denote by $a', b', c'$ the corresponding tangency points)
\[
4\beta^{\tiny \rm symp}_\Omega \left(\frac{1}{3}\right) + \beta^{\tiny \rm out}_\Omega \left(\frac{1}{3}\right) <4 {\rm Area} \ (\triangle a'b'c') + {\rm Area} \ (\triangle A'B'C')=0
\]
and vice versa, violating equality \eqref{beta1/3}. 
Thus, all orbits are action minimizing and therefore they must belong to an invariant curve. 
}

\end{proof}

We can state the following result.

\begin{theorem}\label{thm:beta1/3}
Let $\Omega$ be such that $B^{\tiny \rm out}_\Omega$  admits an invariant curve consisting of periodic points of period $3$. Then:
\begin{equation}\label{ineqouter13}
	\beta^{\tiny \rm out}_\Omega \left(\frac 13\right) \geq  \frac{|\Omega|}{\pi} \beta^{\tiny \rm out}_{\D}\left(\frac 13\right),
	\end{equation}
where $|\Omega|$ denotes the area of $\Omega$ and  $\D$ the unit disk. 
 In particular, equality holds if and only if $\Omega$ is an ellipse.\\
\end{theorem}

\begin{remark}
Therefore all $\Omega$ constructed by Genin-Tabachnikov in \cite{genin-tabachnikov}
(see Remar \ref{reminvcurvesperiod3}) provide examples for which inequality \eqref{ineqouter13} is strict.
\end{remark}

\medskip

\begin{proof}
Using Corollary \ref{propouter13},  Theorem \ref{thmsymplbill}, Examples \ref{example2} and \ref{example3}, we get
$$
\beta^{\tiny \rm out}_\Omega \left(\frac{1}{3}\right)= - 4\,\beta^{\tiny \rm symp}_\Omega \left(\frac{1}{3}\right)
\geq - 4\frac{|\Omega|}{\pi} \beta^{\tiny \rm symp}_{\D} \left(\frac{1}{3}\right)=\frac{|\Omega|}{\pi} \beta^{\tiny \rm out}_{\D}\left(\frac 13\right),
$$
where we used that $\beta^{\tiny \rm out}_{\D}\left(\frac 13\right) = \tan \frac{\pi}{3}  = \sqrt{3}$ and $\beta^{\tiny \rm symp}_{\D}\left(\frac 13\right) = -\frac 12 \sin \frac{2\pi}{3}  =- \sqrt{3}/4$.
In particular, if equality holds, then it follows that $\beta^{\tiny \rm symp}_\Omega (1/3) = \frac{|\Omega|}{\pi} \, \beta^{\tiny \rm sym}_{\D} (1/3)$, hence 
Theorem \ref{thmsymplbill} (ii) implies that $\Omega$ is an ellipse.
\end{proof}

\begin{remark}
Another way to prove Theorem \ref{thm:beta1/3} is to use 
Blaschke inequality \cite{blaschketriangle} for inscribed triangle of maximal area (see Figure \ref{figtriangle})
$$ {\rm Area} ( \triangle{abc})\geq \frac{3\sqrt 3}{4\pi} | \Omega|,$$with equality only for ellipses.
Hence $$
\frac 13{\rm Area}(\triangle ABC)=\frac 43{\rm Area (\triangle abc)\geq \frac{\sqrt 3}{\pi} |\Omega|=\beta^{\tiny \rm out}_{\D}\left(\frac 13\right)\frac{|\Omega|}{\pi}.}
$$
This implies the result.
\end{remark}

\bigskip

\subsection{Rotation number $1/4$} \label{sec:14}
Let us remind first the notion of Radon curves. Radon curves have been thoroughly studied since their introduction more than 100
years ago, we refer to \cite{martini} and \cite{bialybortabachnikov} for modern aspects.\\

Let $\gamma$ be a smooth closed convex curve in the plane, symmetric with respect to the origin. Let
$x,y \in \gamma.$ One says that $y$ is {\it Birkhoff orthogonal} to $x$ if $y$ is parallel to the tangent line
to $\gamma$ at $x$. \\
 This relation is not necessarily symmetric; if it is symmetric, then $\gamma$ is called
a Radon curve. Radon curves comprise a functional space, with ellipses providing a
trivial example.\\

 A characteristic property of $C^2$-Radon curves is the following. A centrally symmetric curve $\gamma$ is a Radon curve if and only if the outer billiard map has an invariant curve consisting  of $4-$periodic orbits.  Each of these orbits form a circumscribed  parallelogram of minimal area (see Figure \ref{figureRadon}). 
\begin{center}
\begin{figure}[h]
	\includegraphics[width=0.4\textwidth]{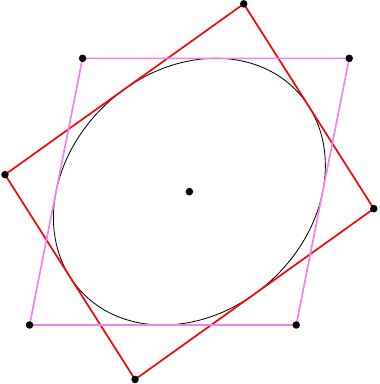}
	\label{figureRadon}
	\caption{Some orbits of period $4$ for a Radon curve}
\end{figure}
\end{center}

\begin{remark}\label{remarkradon}
{\bf (i)} It follows from the properties of $\beta$ function that for centrally symmetric curve Radon property is equivalent to differentiability of $\beta$ at $1/4$. \\
{\bf (ii)} Moreover, if one normalize the area of these parallelograms to be $4$, then 
Radon curve is symplectically self dual, that is it coincides with its polar rotated by $\pi/2$.
\end{remark}

\begin{theorem} \label{thm:14centralsym}
Let $\gamma=\partial \Omega$  be a centrally symmetric Radon curve. Then
\begin{equation}\label{ineqouter14}
\beta^{\tiny \rm out}_\Omega \left(\frac 14\right) \geq  \frac{|\Omega|}{\pi} \beta^{\tiny \rm out}_{\D}\left(\frac 14\right)
\end{equation}
with the equality if and only if $\Omega$ is an ellipse.
\label{thm:betaRadon}
\end{theorem}

\begin{remark}
Therefore all non-elliptic Radon curves provide  examples for which inequality \eqref{ineqouter14} is strict.
\end{remark}

\begin{proof}
By	rescaling, we can assume that the  area of circumscribed parallelograms equals 4. In this case curve the curve
	$\gamma$ coincides with its polar dual rotated by $\pi/2$ (see Remark \ref{remarkradon} (ii)). Hence, by Blaschke--Santalo inequality  (see for instance \cite{santalo}) we obtain that
	$
	|\Omega| \leq \pi.
	$	
Therefore:
\begin{eqnarray*}
\beta^{\tiny \rm out}_\Omega \left(\frac 14\right) = 1 
\geq  \frac{|\Omega|}{\pi} \beta^{\tiny \rm out}_{\D_1}\left(\frac 14\right),
\end{eqnarray*}
where we have used that 
$\beta^{\tiny \rm out}_{\D}\left(\frac 14\right) = \tan \frac{\pi}{4} =1$, see Example \ref{example2}.
In particular, if equality holds, then Blaschke--Santalo inequality is an equality, hence $\Omega$ must be an ellipse. 
	\end{proof}

\subsubsection{Relaxing central symmetry assumption in Theorem \ref{thm:betaRadon}}

Let us start with the following geometric observation, whose proof is obvious (see Figure \ref{figurequadrilateral}).

\begin{lemma}\label{lemmaquadrilater}
Let $A, B, C, D$ the vertices a convex quadrilateral and denote by $a,b,c,d$ the midpoints on each side. Then, the quadrilateral obtained by joining $a,b,c,d$ is a parallelogram, whose area is half of the area of the original quadrilateral.
Moreover, if $ABCD$ is a parallelogram, then the diagonal of $abcd$ are parallel to the sides of $ABCD$.
\end{lemma}

\begin{center}
\begin{figure} 
\tikzset{every picture/.style={line width=0.5pt}} 
\begin{tikzpicture}[x=0.75pt,y=0.75pt,yscale=-0.75,xscale=0.75]
\draw   (151,23) -- (475,89) -- (338,270) -- (156,201) -- cycle ;
\draw    (153.5,112) -- (247,235.5) ;
\draw  [dash pattern={on 0.84pt off 2.51pt}]  (151,23) -- (338,270) ;
\draw    (322,59) -- (403,186) ;
\draw    (153.5,112) -- (322,59) ;
\draw    (247,235.5) -- (403,186) ;
\draw  [dash pattern={on 0.84pt off 2.51pt}]  (156,201) -- (475,89) ;

\draw (137,10) node [anchor=north west][inner sep=0.75pt]  [xscale=0.5,yscale=0.5] [align=left] {A};
\draw (137,198) node [anchor=north west][inner sep=0.75pt]  [xscale=0.5,yscale=0.5] [align=left] {B};
\draw (335,272) node [anchor=north west][inner sep=0.75pt]  [xscale=0.5,yscale=0.5] [align=left] {C};
\draw (479,82) node [anchor=north west][inner sep=0.75pt]  [xscale=0.5,yscale=0.5] [align=left] {D};
\draw (133,102) node [anchor=north west][inner sep=0.75pt]  [xscale=0.5,yscale=0.5] [align=left] {a};
\draw (237,237) node [anchor=north west][inner sep=0.75pt]  [xscale=0.5,yscale=0.5] [align=left] {b};
\draw (408.5,182.5) node [anchor=north west][inner sep=0.75pt]  [xscale=0.5,yscale=0.5] [align=left] {c};
\draw (320.5,36.5) node [anchor=north west][inner sep=0.75pt]  [xscale=0.5,yscale=0.5] [align=left] {d};
\end{tikzpicture}
\label{figurequadrilateral}
\caption{Geometric construction in Lemma \ref{lemmaquadrilater}}
\end{figure}
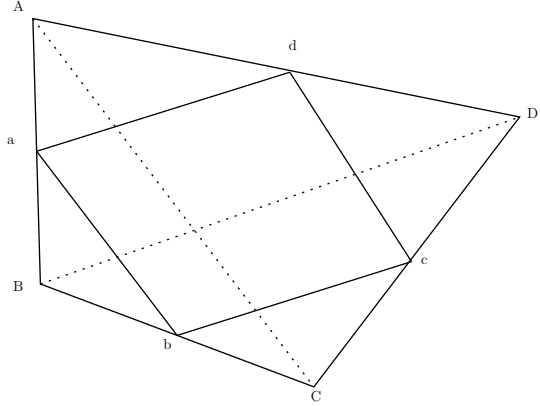
\end{center}

\medskip

\begin{corollary}\label{remarkparallelogram}
It follows from Lemma \ref{lemmaquadrilater} that if an orbit of  $B^{\tiny \rm out}_\Omega$ corresponds to a parallelogram $ABCD$, then the quadrilateral constructed by joining the midpoints of each side  (as in Lemma \ref{lemmaquadrilater}) corresponds to an orbit of $B^{\tiny \rm sym}_\Omega$, which  is also a  parallelogram and its area equals half the area of $ABCD$ .\\
\end{corollary}

\begin{proposition}
The following inequality holds:
\begin{equation}\label{inequalitybeta14}
\beta^{\tiny \rm out}_\Omega \left(\frac{1}{4}\right) + 2\,\beta^{\tiny \rm symp}_\Omega \left(\frac{1}{4}\right) \leq 0.
\end{equation}
Moreover, if equality holds in (\ref{inequalitybeta14}), then  both $B^{\tiny \rm out}_\Omega$ and  $B^{\tiny \rm symp}_\Omega$ admit  invariant curves consisting of periodic points of period $4$.
\end{proposition}

\begin{proof}
Let $A,B,C,D$ be the vertices of a quadrilateral corresponding to an orbit of $B^{\tiny \rm out}_\Omega$ and construct the inscribed polygon with vertices at the tangency points ({\it i.e.}, the midpoints of each side), which a-priori might not be an orbit of $B^{\tiny \rm symp}_\Omega$. We denote by ${\rm Area}(\square ABCD)$ and ${\rm Area}(\square abcd)$ the respective areas and by Lemma \ref{lemmaquadrilater},
$$
{\rm Area}(\square ABCD) = 2\, {\rm Area}(\square abcd).
$$

In particular, considering a minimizing orbit $A,B,C,D$ we obtain 
\begin{equation}
\beta^{\tiny \rm out}_\Omega \left(\frac{1}{4}\right) =\frac{{\rm Area}(\square ABCD)}{4} 
= \frac{{\rm Area}(\square abcd)}{2} 
\leq
 -  2  \beta^{\tiny \rm symp}_\Omega \left(\frac{1}{4}\right),
\end{equation}
where we used that $\beta^{\tiny \rm symp}_\Omega \left(\frac{1}{4}\right) \leq - \frac{1}{4} {\rm Area}(\square abcd)$. This proves the claimed inequality.

Now, let us assume that equality holds in \eqref{inequalitybeta14}.
Let us first show that $B^{\tiny \rm out}_\Omega$ cannot admit a periodic orbit of period $4$ that is not action-minimizing (this implies that it admits an invariant curve consisting of periodic orbits of period $4$). By contradiction, suppose that $A,B,C,D$ are the vertices of a quadrilateral of area
$
{\rm Area}(\square ABCD) >  4\, \beta^{\tiny \rm out}_\Omega \left(\frac{1}{4}\right).
$
Considering the associated  inscribed parallelogram with vertices $a, b, c, d$ as above and using that equality holds in \eqref{inequalitybeta14}, we deduce that
$$
{\rm Area}(\square abcd) = \frac 12 \,{\rm Area}(\square ABCD) \;>\;  2\, \beta^{\tiny \rm out}_\Omega \left(\frac{1}{4}\right) \\
=   - 4\, \beta^{\tiny \rm symp}_\Omega \left(\frac{1}{4}\right),
$$
which leads to the contradiction:
$$ - \frac{{\rm Area}(\square abcd)}{4} < \beta^{\tiny \rm symp}_\Omega \left(\frac{1}{4}\right).
$$
Therefore $B^{\tiny \rm out}_\Omega$ must admit an invariant curve consisting of periodic orbits of period $4$. In particular, each of this orbits determines an inscribed parallelogram of minimal action for
$B^{\tiny \rm symp}_\Omega$, hence also $B^{\tiny \rm symp}_\Omega$ admits an invariant curve consisting of periodic points of period $4$.

\end{proof}

\medskip

\begin{theorem}\label{thm:nonsymmetric1/4}
Let $\Omega$ be such that both $B_{\Omega}^{\tiny\mathrm{out}}$ 
and $B_{\Omega}^{\tiny\mathrm{sym}}$
admit invariant curves consisting of periodic points of period 4.  Then 
$$	\beta^{\tiny \rm out}_\Omega \left(\frac 14\right) \geq  \frac{|\Omega|}{\pi} \beta^{\tiny \rm out}_{\D}\left(\frac 14\right),$$
where $|\Omega|$ denotes the area of $\Omega$ and $\D$ the unit disk.  Moreover, equality holds if and only if $\Omega$ is an ellipse.\\
\end{theorem}

\begin{proof}
Using  (\ref{inequalitybeta14}) and  Theorem \ref{thmsymplbill}, we get:
$$
\beta^{\tiny \rm out}_\Omega \left(\frac{1}{4}\right) = - 2\,\beta^{\tiny \rm symp}_\Omega \left(\frac{1}{4}\right)
\geq - 2\,\frac{|\Omega|}{\pi} \, \beta^{\tiny \rm symp}_{\D} \left(\frac{1}{4}\right)=$$
$$
= - \frac{2\,|\Omega|}{\pi} \left(
- \frac{1}{2} \sin \left(\frac{\pi}{2}\right)
\right)\\
= \frac{|\Omega|}{\pi} 
 \;=\; \frac{|\Omega|}{\pi} \, \beta^{\tiny \rm out}_{\D}\left(\frac 14\right),\\
$$
where we have used that $\beta^{\tiny \rm out}_{\D}\left(\frac 14\right) = \tan \frac{\pi}{4} = 1$.\\
In particular, if equality holds, then it follows that $\beta^{\tiny \rm symp}_\Omega (1/4) = \frac{|\Omega|}{\pi} \, \beta^{\tiny \rm symp}_{\D} (1/4)$, hence it follows from 
Theorem \ref{thmsymplbill} (ii) that $\Omega$ is an ellipse.
\end{proof}

\medskip

\begin{remark} Theorem \ref{thm:nonsymmetric1/4} gives a generalization of Theorem \ref{thm:betaRadon} since no centrall symmetry is not required in Theorem \ref{thm:nonsymmetric1/4}.
It would be interesting to find examples of non centrally symmetric $\Omega$ satisfying assumptions of Theorem \ref{thm:nonsymmetric1/4}.\\
\end{remark}

\medskip

We conclude this section by observing that the assumption on the existence of invariant curve is essential in Theorems \ref{thm:betaRadon} and  \ref{thm:nonsymmetric1/4}.
\begin{example}
Take a unit disk $\mathcal{D}$, and squeeze it to get a domain $\Omega$ of the same area as $\mathcal{D}$, as in the picture. Notice that the minimal area of the circumscribed quadrilateral is decreased. Hence, for such a domain $\Omega$ we have
\begin{equation}
	\label{eq:counterexamples}
	\beta^{\tiny \rm out}_\Omega \left(\frac 14\right) <  \frac{|\Omega|}{\pi} \beta^{\tiny \rm out}_{\D}\left(\frac 14\right).
\end{equation}
Obviously, one can get in this way also a centrally symmetric $\Omega$. 
\begin{figure}[h]
	\centering	
	\tikzset{every picture/.style={line width=0.75pt}} 
	\begin{tikzpicture}[x=0.75pt,y=0.75pt,yscale=-1,xscale=1]
		\draw  [dash pattern={on 4.5pt off 4.5pt}] (137,130.17) .. controls (137,87) and (172,52) .. (215.17,52) .. controls (258.34,52) and (293.33,87) .. (293.33,130.17) .. controls (293.33,173.34) and (258.34,208.33) .. (215.17,208.33) .. controls (172,208.33) and (137,173.34) .. (137,130.17) -- cycle ;
		\draw   (137,52) -- (293.33,52) -- (293.33,208.33) -- (137,208.33) -- cycle ;
		\draw [line width=1.5]    (137.35,130.86) .. controls (136.33,98.67) and (155.33,61.67) .. (174.33,59.67) .. controls (193.33,57.67) and (249.33,54.67) .. (263.33,61.67) .. controls (277.33,68.67) and (292.33,91.67) .. (293.33,131.17) ;
		\draw  [draw opacity=0][line width=1.5]  (293.33,130.86) .. controls (293.33,130.87) and (293.33,130.88) .. (293.33,130.89) .. controls (293.34,173.96) and (258.43,208.89) .. (215.35,208.9) .. controls (172.28,208.9) and (137.35,173.99) .. (137.35,130.92) .. controls (137.35,130.9) and (137.35,130.88) .. (137.35,130.86) -- (215.34,130.9) -- cycle ; \draw  [line width=1.5]  (293.33,130.86) .. controls (293.33,130.87) and (293.33,130.88) .. (293.33,130.89) .. controls (293.34,173.96) and (258.43,208.89) .. (215.35,208.9) .. controls (172.28,208.9) and (137.35,173.99) .. (137.35,130.92) .. controls (137.35,130.9) and (137.35,130.88) .. (137.35,130.86) ;  
		\draw [color={rgb, 255:red, 208; green, 2; blue, 27 }  ,draw opacity=1 ]   (137,56) -- (293,57) ;

	\end{tikzpicture}
	\caption{A convex set $\Omega$ satisfying \eqref{eq:counterexamples}.}
\end{figure}
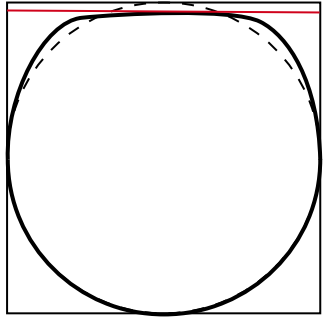
\end{example}

\subsection{Some open questions} \label{sec:openquestions}

We now formulate two questions that arise quite naturally from the previous discussion.\\

\noindent {\bf Question 1.} {\it Can one claim results similar to those in Theorems \ref{thm:beta1/3}, \ref{thm:betaRadon} and  \ref{thm:nonsymmetric1/4} for rational rotation numbers different from $1/3$ and $1/4$, under the assumption that there exists of an invariant curve of that rotation number and consisting of periodic points?}\\

Inspecting the proof of Theorems \ref{thm:beta1/3} and \ref{thm:nonsymmetric1/4}, it seems quite natural to frame this question in terms of area deviation. It is possible to show (see \cite[Section 2.4]{Fejesetco}) that, for any $n\geq 3$ there exists a circumscribed $n$-gon $\mathcal{P}_n$ and and inscribed one $\mathcal{R}_n$ satisfying 
\[
\frac{\vert \mathcal{P}_n\vert-\vert\mathcal{R}_n\vert}{\vert\mathcal{P}_n\vert} \le \sin^2 \left(\frac{\pi}{n}\right).
\]
For $n=3,4$ Proposition \ref{propouter13} and Lemma \ref{lemmaquadrilater} show that this inequality is actually an equality for all outer billiard trajectories, independently on the shape of $\Omega$. We thus wonder:\\

\noindent {\bf Question 2.} {\it For which $\Omega$ there exists $n$-gons $\mathcal{P}_n$ and $\mathcal{R}_n$ 
such that
\begin{equation}
	\label{ratio_areas_equality}
\frac{\vert \mathcal{P}_n\vert-\vert\mathcal{R}_n\vert}{\vert\mathcal{P}_n\vert} = \sin^2 \left(\frac{\pi}{n}\right)?
\end{equation}
}

\medskip

For such an $\Omega$ we would have that
\begin{equation}
	\label{inequality_beta_outer}
\beta_{\Omega}^{\mathrm{out}}\left(\frac{1}{n}\right) \ge \frac{\vert \mathcal{P}_n\vert}{n}=\frac{\vert \mathcal{R}_n\vert}{n \cos^2\left(\frac{\pi}{n}\right)} \ge-\beta^{\mathrm{symp}}_\Omega\left(\frac{1}{n}\right)\sec^2\left(\frac{\pi}{n}\right).
\end{equation}
If in \eqref{inequality_beta_outer} equalities hold, both the vertices of $\mathcal{P}_n$ and $\mathcal{R}_n$ correspond  to minimizing configutations. If equality \eqref{ratio_areas_equality} holds for all billiard trajectories, $\Omega$ possesses an invariant curve if and only if \eqref{inequality_beta_outer} is an equality. Moreover, combining with Theorem \ref{thmsymplbill}, we conclude that for any $\Omega$ satisfying \eqref{ratio_areas_equality} 
\[
\beta_{\Omega}^{\mathrm{out}}\left(\frac{1}{n}\right) \ge -\frac{\vert \Omega \vert }{\pi} \beta^{\mathrm{symp}}_{\mathcal{D}}\left(\frac{1}{n}\right)\sec^2\left(\frac{\pi}{n}\right)  = \frac{\vert \Omega\vert}{\pi} \beta_{\mathcal{D}}^{\mathrm{out}}\left(\frac{1}{n}\right),
\]
with equality if and only if $\Omega$ is an ellipse, {\it i.e.}, the same type of estimate as in Theorems \ref{thm:beta1/3} and \ref{thm:nonsymmetric1/4}.

\end{document}